\documentclass[11pt]{article}

\date{}

\author{}
\usepackage{amsmath,amsthm,amsfonts,amssymb}

\usepackage{graphicx}

\newcommand{\captionfonts}{\footnotesize}
\makeatletter  
\long\def\@makecaption#1#2{%
  \vskip\abovecaptionskip
  \sbox\@tempboxa{{\captionfonts #1: #2}}%
  \ifdim \wd\@tempboxa >\hsize
    {\captionfonts #1: #2\par}
  \else
    \hbox to\hsize{\hfil\box\@tempboxa\hfil}%
  \fi
  \vskip\belowcaptionskip}
\makeatother   

\newcommand{\nwc}{\newcommand}

\newtheorem{prop}{Proposition}[section]
\newtheorem{lemma}[prop]{Lemma}

\newtheorem{theorem}[prop]{Theorem}
\newtheorem{corollary}[prop]{Corollary}

\nwc{\R}{\mathbb R}
\nwc{\Z}{\mathbb Z}
\nwc{\N}{\mathbb N}

\newcommand{\ignore}[1]{}

\nwc{\eps}{\varepsilon}
\nwc{\re}{Re\,}

\nwc{\wto}{\rightharpoonup}

\nwc{\ds}{\displaystyle}
\newcommand {\bedis} {\begin{displaymath}}
\newcommand {\edis} {\end{displaymath}}
\newcommand{\newbeqna} {\renewcommand {\arraystretch} {2}
                        \begin {displaymath} \begin {array}{crcl}}
\newcommand{\neweqna}{\end{array} \end {displaymath}}

\newcommand{\fbeqna}{\renewcommand {\arraystretch} {1.3}
\begin {displaymath}\begin{array}{rcll}}
\newcommand{\feqna}{\end{array}\end{displaymath}}

\newcommand {\beqna} {\begin{eqnarray*}}
\newcommand {\eqna} {\end{eqnarray*}}
\newcommand {\beqn} {\begin{eqnarray}}
\newcommand {\eqn} {\end{eqnarray}}

\begin{document}
\title{Exponential tail behaviour of self-similar solutions to Smoluchowski's coagulation 
equation}\author{%
B. Niethammer%
\footnote{Institute of Applied Mathematics,
University of Bonn, Endenicher Allee 60, 53115 Bonn, Germany}
and
J. J. L. Vel\'{a}zquez\footnote{Institute of Applied Mathematics,
University of Bonn, Endenicher Allee 60, 53115 Bonn, Germany}
}

\maketitle
\begin{abstract}
We consider self-similar solutions with finite mass to Smoluchowski's coagulation equation  for rate kernels that have homogeneity zero but are possibly singular such
as Smoluchowski's original kernel. We prove pointwise exponential decay of these solutions under rather mild assumptions on the kernel.  If the support of the kernel is sufficiently 
large around the diagonal we also proof that $\lim_{x \to \infty} \frac{1}{x} \log \Big( \frac{1}{f(x)}\Big)$ exists. In addition we prove properties of the prefactor if
the kernel is uniformly bounded below.
\end{abstract}

\section{Introduction}
\label{S.introduction}
In this article we study the decay properties of self-similar solutions to Smoluchowski's classical coagulation equation \cite{Smolu16}. This model describes the binary collision of clusters 
of size $\xi >0$ where the rate of coagulation of clusters of size $\xi$ and $\eta$ is governed by a nonnegative and symmetric rate kernel $K=K(\xi,\eta)$. Then the number density $n(t,\xi)$ 
of clusters of size $\xi$ at time $t$ is determined by the rate equation
\begin{equation}\label{smolu1}
 \partial_t n(t,\xi)=\frac 1 2 \int_0^\xi K(\eta,\xi-\eta) n(t,\xi-\eta)n(t,\eta)\,d\eta - \int_0^{\infty} K(\xi,\eta) n(t,\xi)n(t,\eta)\,d\eta \,.
\end{equation}
 The precise form of $K$ is determined by the microscopic
details of the coagulation process. In \cite{Smolu16} Smoluchowski derived the kernel $K(\xi,\eta)=K_0\big(\xi^{1/3} + \eta^{1/3}\big) \big( \xi^{-1/3} + y^{-1/3}\big)$ in the case that clusters are approximately balls in three
dimensions, diffuse by Brownian motion when they are well separated and coagulate quickly when they get close.  The coagulation equation \eqref{smolu1} has since been used in  many different application areas, such as
aerosol physics, astronomy or population dynamics, and correspondingly one finds a large variety of different kernels in the respective literature (see e.g. \cite{Drake72,Aldous99}).

It is well-known by now that the behaviour of solutions to \eqref{smolu1} depends on the growth behaviour of $K$ for large clusters. More precisely, solutions to \eqref{smolu1} only conserve
mass for all times if the kernel $K$ grows at most linearly at infinity. If $K$ grows faster than linearly at infinity, the phenomenon of gelation occurs, that is, roughly speaking, infinitely large clusters are 
created at finite time and as a consequence, there is a finite time after which the first moment of $n$ decreases  \cite{McLeod62a,Jeon98,EMP02}.

In this article we are not concerned with the gelation phenomenon, but will consider kernels that grow slower than linearly. In this case it has been proved for a large class of kernels \cite{Norris99,LauMisch02,LauMisch04}
that for initial data $n_0$ with finite first moment there exists a unique global solution to \eqref{smolu1} that conserves the mass for all times.

A topic of particular interest in the analysis of coagulation equations is the long-time behaviour of solutions. Since most kernels $K$ are homogeneous one expects due to a corresponding
 scale invariance of equation \eqref{smolu1}, that solutions of  \eqref{smolu1} converge to a unique self-similar solution as $t \to \infty$.  This so-called scaling hypothesis is however so far only well-understood for the
so-called solvable kernel, that is for $K=const., K=\xi+\eta$ and $K=\xi\eta$, for which explicit solution formulas are available. A complete analysis of the large-time behaviour of solutions for those
kernels has been given in \cite{MePe04}. More precisely, it has been shown that \eqref{smolu1} has a unique self-similar solution with finite mass and in addition a family of self-similar solutions with infinite mass,
so-called fat-tail solutions. The domains of attraction of each of these self-similar solutions has been characterized. Roughly speaking, initial data with a certain decay behaviour at infinity produce
solutions that converge to the self-similar solution with the same decay behaviour.

For non-solvable kernels much less is understood about the large-time behaviour and even to understand self-similar solutions turns out to be a formidable task. Only quite recently, the first existence
results of self-similar solutions for homogeneous kernels with degree $\gamma<1$ have become available \cite{FouLau05,EMR05} and certain properties of the corresponding self-similar profiles
have been established \cite{FouLau06a,EsMisch06,CanMisch11,NV11b}. Also, the existence of self-similar solutions with fat tails has been proved for a range of kernels \cite{NV11a,NV12a}, but one of the
major open problems is still the uniqueness of self-similar solutions either with finite mass, or, in the case of fat tails,  with a given power-law decay behaviour. One idea to prove uniqueness has been to 
characterize the asymptotics
of self-similar solutions for small cluster sizes in order to be able to reformulate the uniqueness problem as an initial value problem. However, at least to our knowledge, this strategy has
not yet been successful. In recent work we were able to prove the first uniqueness result for self-similar solutions with finite mass for kernels that are in a certain sense close to the constant one 
\cite{NV13b}.
One key ingredient of the proof is a precise understanding of the decay behaviour of solutions at infinity which is one of the results of the present paper.
In addition, we also prove here more detailed estimates for self-similar solutions (for kernels with homogeneity zero) that are of interest by itself but might also be helpful in proving uniqueness of self-similar
solutions for kernels that are not necessarily close to the constant one. 

In order to formulate our  results we first recall the equation for self-similar solutions to \eqref{smolu1} with finite mass.
For kernels with homogeneity zero such solutions are of the form $n(t,\xi)=t^{-2} f(x)$ with $x=\xi/t$, and the so-called self-similar profile $f$ satisfies
\begin{equation}\label{eq1}
 -xf'(x) - 2f(x) = \tfrac 1 2 \int_0^x K(x-y,y) f(x-y)f(y)\,dy - f(x) \int_0^{\infty} K(x,y) f(y)\,dy \,
\end{equation}
with
\begin{equation}\label{eq2}
 \int_0^{\infty} x f(x)\,dx =M\,.
\end{equation}
In general however, due to possible divergences of $f$ as $x \to 0$, the integrals in \eqref{eq1} are not necessarily finite such that one has to go over to a weak formulation. 
Under the assumption that $\lim_{x\to 0} x^2 f(x)=0$ one can integrate \eqref{eq1} to obtain that $f$ satisfies
\begin{equation}\label{eq1b}
x^2f(x) = \int_0^x \int_{x-y}^{\infty} K(y,z) y f(z)f(y)\,dz\,dy\,.
\end{equation}
In \cite{vanDoErnst88} the authors gave a self-consistent analysis of  solutions to  \eqref{eq2}, \eqref{eq1b}. More precisely, under the assumption that a nontrivial, nonnegative 
solution  exists, they derived detailed properties of the respective profiles for large and small $x$ for a large range of different kernels.
As indicated above the first rigorous existence results for \eqref{eq2}, \eqref{eq1b} appeared  much later. In \cite{FouLau05} the authors prove for a large range of kernels of homogeneity
$\gamma<1$, that there exists a weak self-similar profile, that is a nonnegative function $f \in L^1(x^2\,dx)$ with $xyK(x,y)f(x)f(y) \in L^1((0,\infty)^2)$ and such that
\eqref{eq1b} is satisfied for almost all $x$. In addition, the self-similar profiles satisfy certain moment estimates that are in agreement with the expectation that profiles with finite mass
decay exponentially.  Another strategy was employed in \cite{EMR05} to establish a similar result but for a smaller class of kernels. In further work, regularity properties
of self-similar profiles were derived, such as in \cite{FouLau06a} for the sum kernel $x^{\lambda}+y^{\lambda}$ with $\lambda \in (0,1)$. More precisely, it was shown that every self-similar profile
$f$ satisfies $f \in C^1(0,\infty)$ and decays exponentially.  For kernels of the form $K(x,y)=x^ay^b+x^by^a$ with $-1 \leq a\leq b<1$ and $a+b \in (-1,1)$ it was proved later in \cite{CanMisch11} that
each self-similar profile is infinitely differentiable.

In what follows we assume that there exists a  function $f \in L^1_{loc}(0,\infty)$ that is nonnegative,  $xf \in L^1(0,\infty)$ and satisfies \eqref{eq1b} for almost all $x>0$. Furthermore, we assume for our convenience
that $f \in C^0(0,\infty)$. Strictly speaking, such a result has been proved only for a smaller class of kernels but it is not difficult to derive, once one has established 
certain properties of the profile as $x \to 0$ such
as Lemma \ref{L.negativemoment} below. One can then proceed as in Chapter 4.2 of \cite{NV12a} to first prove by an iterative argument that $f \in L^{\infty}_{loc}(0,\infty)$ from which continuity follows rather easily. 
 
Our goal in this paper is to  prove several results  on the details of the decay behaviour of such self-similar profiles. The more detailed the results are, the stronger assumptions we need on the kernel $K$. We emphasize, however, that
all but the last result, Theorem \ref{T4}, hold in particular for kernels such as Smoluchowski's original one. We feel that this is important, since many kernels that appear in the applied literature, are of 
a similar form.

Throughout this paper we  assume that $K$ is homogeneous of degree zero  and  satisfies for some $\alpha \in [0,1)$ that
\begin{equation}\label{Kassump1}
 K(x,y) \leq K_0 \Big(\Big(\frac{x}{y}\Big)^{\alpha} + \Big(\frac{y}{x}\Big)^{\alpha}\Big) \qquad \mbox{ for all } x,y \in (0,\infty)
\end{equation}
and for some $\kappa\in (0,1]$ that
\begin{equation}\label{Kassump2}
\min_{|x-y|\leq \kappa(x+y)} K(x,y) \geq k_0>0\,.
\end{equation}

Our first main result establishes pointwise exponential decay of self-similar profiles.
\begin{theorem}\label{T1}
Let $f$ be a solution to \eqref{eq2}-\eqref{eq1b}. Then there exist positive constants  \linebreak
$C_1,C_2,\alpha_1,\alpha_2$ such that
\[
 C_1 e^{-\alpha_1 x} \leq f(x) \leq C_2 e^{-\alpha_2 x}\qquad \mbox{ for all } x \geq 1\,.
\]
\end{theorem}

Theorem \ref{T1} is proved in Section \ref{S.bounds}, more precisely Lemma \ref{L.ublargex} gives the upper bound, while Lemma \ref{L.lowerbound} provides the lower bound.

If we write 
\[f(x)= e^{-xa(x)}, \]
then Theorem \ref{T1} gives that $0<\alpha_1\leq a(x)\leq \alpha_2$. Our next goal is to prove that $\lim_{x \to \infty} a(x)$ exists. 
For this result we need further regularity of the self-similar solution $f$, which is a nontrivial issue since a standard bootstrap argument does not work for equation \eqref{eq1b}. 
Some regularity results are available (cf. \cite{FouLau06a,CanMisch11}) and while the proofs might also apply to the kernels that we consider here, we find it convenient to 
provide a self-consistent proof that $f \in BV_{loc}(0,\infty)$ in Section \ref{S.bv}. 
For this result we  need that  $K$ is sufficiently regular. We assume that $K$ is differentiable and 
\begin{equation}\label{Kdiff}
\Big | \frac{\partial}{\partial x} K(x,y) \Big| \leq C \frac{1}{x} \Big( \Big(\frac{x}{y}\Big)^{\alpha} + \Big(\frac{y}{x}\Big)^{\alpha}\Big)
\qquad \mbox{ for all } x,y>0\,.
\end{equation}
In order to show that $a$ has a limit as $x \to \infty$, we 
also need that the region in which $K$ is bounded below is sufficiently large. More precisely, we have the following result (cf. Proposition \ref{P.limit}).
\begin{theorem}\label{T2}
 Let $K$ satisfy \eqref{Kassump1}, \eqref{Kassump2} with $\kappa >1/3$, as well as \eqref{Kdiff}. Then $\lim_{x \to \infty} a(x)$ exists.
\end{theorem}

We then turn out attention to the prefactor. 
To that aim we define the function $u\colon (0,\infty) \to (0,\infty)$ via 
\begin{equation}\label{udef}
 f(x)=u(x) e^{-a^*x}\qquad \mbox{ where } \qquad a^*=\lim_{x \to \infty} a(x)\,.
\end{equation}
Under the additional assumption that $K$ is uniformly bounded below we  prove the following result in 
Section \ref{S.preex}.
\begin{theorem}\label{T3}
 Let $K$ satisfy \eqref{Kassump1}, \eqref{Kassump2} with $\kappa=1$, as well as \eqref{Kdiff}. Define $u_{R}(x):=u(Rx)$. Then there exists a sequence $R_j \to \infty$ and a 
nonnegative measure $\mu$ such that $u_{R_j} \wto \mu$ on $(0,\infty)$ as $j \to \infty$. The limit $\mu$ is nontrivial and satisfies the equation
\begin{equation}\label{muequation}
 x \mu(x) = \frac{1}{2a^*} \int_0^x K(y,x-y)\mu(x-y)\mu(y)\,dy 
\end{equation}
in the distributional sense (cf. \eqref{weakmuequation}). 
\end{theorem}
Ideally, we would like to know that $\mu$ is uniquely determined (and hence a constant). However, to prove such a result turns out to be difficult and we are currently only able to derive 
such a uniqueness result (as for the full coagulation equation) only for kernels that are close to constant in the sense outlined below.

\begin{theorem}\label{T4}
Let  $K$ satisfy in addition to \eqref{Kassump1}-\eqref{Kassump2} that  
\begin{equation}\label{Kassump4}
 -\eps \leq K-2 \leq \eps \Big( \Big(\frac{x}{y}\Big)^{\alpha} + \Big(\frac{y}{x}\Big)^{\alpha} \Big)
\end{equation}
and 
\begin{equation}\label{Kassump5}
\Big | \frac{\partial}{\partial x} K(x,y) \Big| \leq C\eps \frac{1}{x} \Big( \Big(\frac{x}{y}\Big)^{\alpha} + \Big(\frac{y}{x}\Big)^{\alpha}\Big)
\qquad \mbox{ for all } x,y>0\,.
\end{equation}
Then, if $\eps>0$ is sufficiently small, the only nontrivial solution to \eqref{muequation} is 
the constant one, i.e. $\mu \equiv 2a^*/\int_0^1 K(s,1-s)\,ds$.
\end{theorem}
The proof of Theorem \ref{T4} is contained in Section \ref{S.preuni}.

\section{Estimates for large $x$}\label{S.bounds}

\subsection{An upper bound}

\begin{lemma}
 \label{L.ubsmallxnew}
There exists a constant $C>0$ such that any solution of \eqref{eq2}, \eqref{eq1b} satisfies
\begin{equation}\label{ubsmallx1}
 \sup_{R>0}\frac{1}{R} \int_{R/2}^Rx f(x)\,dx\leq C\,.
\end{equation}
\end{lemma}
\begin{proof}
The proof follows exactly as in Lemma 2.1 of  \cite{NV11b}. 
\end{proof}

\begin{lemma}\label{L.negativemoment}
 For any $\gamma\in[0,1)$ there exists a constant $C>0$ such that any solution of \eqref{eq2}, \eqref{eq1b} satisfies
\begin{equation}\label{ubsmallx}
 \int_0^{\infty} x^{1{-}\gamma} f(x)\,dx \leq C\,.
\end{equation}
\end{lemma}
\begin{proof}
 The estimate in $(0,1)$ follows from \eqref{ubsmallx1}  via a dyadic decomposition, an argument we will also use repeatedly in this paper.
Indeed,
\begin{align*}
\int_0^1 x^{1{-}\gamma} f(x)\,dx & =\sum_{n=0}^{\infty} \int_{2^{-(n{+}1)}}^{2^{-n}} x^{1{-}\gamma} f(x)\,dx\\
&  \leq 2 \sum_{n=0}^{\infty} 2^{n\gamma} \int_{2^{-(n{+}1)}}^{2^{-n}} x f(x)\,dx  \leq C \sum_{n=0}^{\infty} 2^{n(\gamma{-}1)} \leq C\,,
\end{align*}
while the estimate in $(1,\infty)$ is a consequence of  \eqref{eq2}.
\end{proof}

\begin{lemma}
 \label{L.ublargex}
There exist constants $C,a>0$ such that any solution of \eqref{eq2}, \eqref{eq1b} satisfies
\[
 f(x) \leq C e^{-ax} \qquad \mbox{ for all } x\geq 1\,.
\]
\end{lemma}

\begin{proof}
\ignore{
Dividing \eqref{eq1b} by $x$, integrating and changing the order of integration on the right hand side we derive
as a first a priori estimate that
\begin{align}
 \int_0^{\infty} x f(x)\,dx & = \int_0^{\infty} \int_0^{\infty} K(y,z) f(y) f(z) y \log \Big( \frac{y+z}{y}\Big)\,dy\,dz\nonumber\\
& \geq C \int_0^{\infty}  \int_0^z  f(y) f(z) K(y,z) y \log \Big( \frac{y+z}{y}\Big)\,dy\,dz\label{ublargex1}\\
& \geq C \int_0^{\infty} \int_0^z  f(y) f(z) K(y,z)y\,dy\,dz\,.\nonumber
\end{align}
}
We denote for $\gamma \geq 1$
\begin{equation}\label{ublargex2}
 M(\gamma):=\int_0^{\infty} x^{\gamma} f(x)\,dx \,.
\end{equation}
Our goal is to show inductively that 
\begin{equation}\label{ublargex3}
 M(\gamma) \leq \gamma^{\gamma} e^{A\gamma}
\end{equation}
for some (large) constant $A$.

To that aim we first  multiply \eqref{eq1b} by  $x^{\gamma-2}$ with some $\gamma>1$ and after integrating we obtain 
\[
(\gamma{-}1) M(\gamma) = \tfrac 1 2 \int_0^{\infty}\int_0^{\infty} K(x,y) f(x) f(y)\big( (x+y)^{\gamma} - x^{\gamma} - y^{\gamma} \big)\,dx\,dy\,.
\]
By symmetry we also find
\begin{align*}
 M(\gamma)&= \frac{1}{\gamma{-}1} \int_0^{\infty}  \int_0^x  K(x,y) f(x) f(y) \Big( \big(x+y\big)^{\gamma} - x^{\gamma}\Big) \,dy\,dx\\
& = \int_0^1\,dx \int_0^x \,dy \dots + \int_1^{\infty} \,dx \int_0^{x/\gamma} \,dy \dots + \int_1^{\infty} \,dx \int_{x/\gamma}^x \,dy \dots\,.
\end{align*}
Due to \eqref{ubsmallx} we have
\begin{align}
 \int_0^1\int_0^x  K(x,y) f(x) f(y)& \Big( \big(x+y\big)^{\gamma} - x^{\gamma}\Big)\,dy\,dx
\leq C \int_0^1  \int_0^x  K(x,y) f(x) f(y)  x^{\gamma}\,dy\,dx\label{ublargex4} \\
& \leq C\int_0^1 x^{\alpha +\gamma}f(x)\int_0^x y^{1{-}\alpha} f(y)\,dy\,dx \leq C\,.\nonumber
\end{align}
Using \eqref{ubsmallx} and  $(x+y)^{\gamma} - x^{\gamma} \leq c x^{\gamma-1}y$ for $ y \leq \frac{x}{\gamma}$, we find that
\begin{equation}\label{ublargex5}
\begin{split}
 \int_1^{\infty}  \int_0^{x/\gamma} K(x,y) f(x) f(y)& \Big( \big(x+y\big)^{\gamma} - x^{\gamma}\Big)\,dy\,dx\\
& \leq  \int_1^{\infty}  \int_0^{x/\gamma} K(x,y)  y x^{\gamma{-}1} f(x)f(y)\,dy\,dx\\
&\leq \int_1^{\infty} \int_0^{x/\gamma} x^{\gamma+\alpha-1} y^{1-\alpha} f(x)f(y)\,dy\,dx 
\\ &\leq
 C M(\gamma+\alpha -1)\,,
\end{split}
\end{equation}
so that for the sum of both terms we can prove by induction that it is smaller than $1/2 \gamma^{\gamma} e^{A\gamma}$. 
It remains to estimate
\begin{align*}
\frac{C}{\gamma-1}&\int_1^{\infty} \,dx \int_{x/\gamma}^x\,dy   K(x,y) f(x) f(y) \big(x+y\big)^{\gamma}\\
& \leq \frac{C}{\gamma} \int_1^{\infty} \,dx \int_{x/\gamma}^x \,dy f(x) f(y) \big(x+y\big)^{\gamma} \Big( \frac{x}{y}\Big)^{\alpha}=:(*)\,.
\end{align*}
In the following $\{\zeta_n\} \subset (0,1]$, $\zeta_0=1$,  will be a decreasing sequence of numbers that will be specified later.
Then we define a corresponding sequence of numbers $\kappa_n$ such that given a sequence $\{\theta_n\} \subset (0,1)$, also to be
specified later, we have
\begin{equation}\label{kappandef}
 \big(x+y\big)^{\gamma}  \leq \kappa_n^{\gamma} x^{\gamma(1-\theta_n)} y^{\gamma \theta_n} \qquad \mbox{ for } \frac{y}{x} \in [\zeta_{n+1},\zeta_n]\,.
\end{equation}
Equivalently we have
\begin{equation}\label{kappadef1}
 \kappa_n = \max_{\zeta \in [\zeta_{n{+}1},\zeta_n]} \Big( \frac{1+\zeta}{\zeta^{\theta_n}}\Big)\,.
\end{equation}
With these definitions we have
\[
 (*) \leq \frac{C}{\gamma} \sum_{n=0}^{n_0(\gamma)} \kappa_n^{\gamma} \zeta_{n{+}1}^{-\alpha} M(\gamma(1-\theta_n)) M(\gamma \theta_n)\,,
\]
where $n_0(\gamma)$ is such that $\zeta_{n_0(\gamma)} = \frac{1}{\gamma}$. 

We now choose $\theta_n$ such that for $\psi_{\theta_n}(\zeta):= \log(1+\zeta)-\theta_n \log \zeta$ we have
\[
\min_{\zeta \in [\zeta_{n+1},\zeta_n]} \psi_{\theta_n}(\zeta) = \log(1+\zeta_n) - \theta_n \log(\zeta_n)\,.
\]
This is equivalent to 
\begin{equation}\label{thetandef}
\theta_n = \frac{\zeta_n}{1+\zeta_n}\,.
\end{equation}
We want to prove  by induction over $\gamma$ that $(*) \leq \frac{1}{2} \gamma^{\gamma} e^{A\gamma}$. Inserting the corresponding hypothesis, this reduces to showing that
\[
 \frac{C}{\gamma} \sum_{n=0}^{n_0} \zeta_{n{+}1}^{-\alpha} \exp \Big( \gamma \big( \max_{\zeta \in [\zeta_{n+1},\zeta_n]} \psi_{\theta_n}(\zeta)
+ \theta_n \log \theta_n + (1-\theta_n) \log (1-\theta_n)\big) \Big)\leq \frac 1 2\,.
\]
By definition \eqref{thetandef} we have
\begin{align*}
 &\max_{\zeta \in [\zeta_{n+1},\zeta_n]} \psi_{\theta_n}(\zeta)
+ \theta_n \log \theta_n + (1-\theta_n) \log (1-\theta_n) \\
& = \min_{\zeta \in [\zeta_{n+1},\zeta_n]}\psi_{\theta_n}(\zeta)  + (\max-\min)_{\zeta \in [\zeta_{n+1},\zeta_n]}  \psi_{\theta_n}(\zeta)
+ \theta_n \log \theta_n + (1-\theta_n) \log (1-\theta_n) \\
& = (\max-\min)_{\zeta \in [\zeta_{n+1},\zeta_n]}  \psi_{\theta_n}(\zeta)\,.
\end{align*}
Thus we need to investigate
\begin{align*}
 (\max-\min)_{\zeta \in [\zeta_{n+1},\zeta_n]}  \psi_{\theta_n}(\zeta) &= \psi_{\theta_n}(\zeta_{n{+}1}) - \psi_{\theta_n}(\zeta_n)\\
& = \log \Big( \frac{1+\zeta_{n{+}1}}{1+\zeta_n} \Big) - \theta_n \log \Big( \frac{\zeta_{n{+}1}}{\zeta_n} \Big)\\
& = \log \Big( 1 + \frac{\zeta_{n{+}1}-\zeta_n}{1+\zeta_n}\Big) - \frac{\zeta_n}{1+\zeta_n} \log \Big( 1 + \frac{\zeta_{n{+}1}-\zeta_n}{\zeta_n}\Big)\,.
\end{align*}
Expanding the nonlinear terms we find
\begin{equation*}
W:=(\max-\min)_{\zeta \in [\zeta_{n+1},\zeta_n]}  \psi_{\theta_n}(\zeta) \leq C \Big( |\zeta_{n{+}1}-\zeta_n|^2 + \frac{(\zeta_{n{+}1}-\zeta_n\big)^2}{
\zeta_n}\Big)\,.
\end{equation*}
We now split $\{1,2,\cdots,n_0\}$ into a finite number of sets $\{1,2,\cdots,N_1\}$, $\{N_1+1, \cdots, N_2\}$, $\cdots$, $\{N_{k-1}+1, \cdots,N_k=n_0\}$ in the following way.

We first define 
\[
 \zeta_0 =1 \,, \quad \eta_0= 1+ \frac{1}{\sqrt{\gamma}}\,, \quad \zeta_n = \eta_0^{-n} \zeta_0\,, \quad \mbox{ for all } n \leq N_1
\]
where $N_1$ is such that $ \zeta_{n} \geq \frac{1}{\sqrt{\gamma}}$, that is we can choose $N_1 \sim \sqrt{\gamma} \log \gamma$.
With these definitions we find 
\[
 \Big| \frac{\zeta_{n{+}1} - \zeta_n}{\zeta_n}\Big| \leq \frac{C}{\sqrt{\gamma}} \qquad \mbox{ for all } 1\leq n\leq N_1
\]
and thus
\[
 \frac{1}{\gamma} \sum_{n=0}^{N_1} \zeta_{n{+}1}^{-\alpha} \exp \big( \gamma W\big) \leq \frac{C N_1}{\gamma} \gamma^{\alpha/2}  
\sim \gamma^{(\alpha-1)/2} \log \gamma \to 0 \quad \mbox{ as } \gamma \to \infty\,.
\]
Next, for $n \in (N_1,N_2]$ we define
\[
 \zeta_n=\eta_1^{-(n-N_1)} \zeta_{N_1}\,, \qquad \eta_1=1+\frac{1}{\gamma^{1/4}}\,.
\]
Then $\zeta_n \leq \frac{1}{\sqrt{\gamma}}$, such that $|\zeta_{n{+}1}-\zeta_n|^2 \leq \frac{C}{\gamma}$ and 
\[
\gamma W \leq C \gamma \Big( \zeta_n |\eta_1-1|^2 + \frac{1}{\gamma}\Big) \leq C\,.
\]
We need to determine $N_2$ such that
\[
(N_2-N_1) \frac{1}{\zeta_{N_2}^{\alpha} \gamma} \to 0 \qquad \mbox{ as } \gamma \to \infty\,.
\]
Making the ansatz $\zeta_{N_2} = \gamma^{-\sigma}$, that is $N_2 \sim \gamma^{1/4} \log \gamma$, this implies that we need
\[
 \frac{\gamma^{1/4} \gamma^{\alpha \sigma} \log \gamma}{\gamma} \ll 1\,
\]
as $\gamma \to \infty$ 
and this needs $\alpha \sigma < \frac 3 4$.
If we can choose $\sigma =1$ we are done. Otherwise we need to iterate the argument. Thus, assume that $3\leq 4\alpha$ and set
$\sigma_1=\frac{3}{8\alpha}$. We define for 
 $k \geq 1$ the sequence $\sigma_{k+1} = \frac{1}{4\alpha} (1+\sigma_k)$.
Then, we define $\eta_{k+1}= 1+\frac{1}{\gamma^{(1-\sigma_{k+1})/2}}$ and $\zeta_n = \eta_{k+1}^{-(n-N_{k+1})} \zeta_{N_{k+1}}$ for $n \in (N_{k+1},N_{k+2}]$ with 
$\zeta_{N_{k+1}}=\gamma^{-\sigma_k}$. Then $N_{k+1}-N_k = \gamma^{(1-\sigma_k)/2} \log \gamma $ and we find that our sum is controlled by
$C \gamma^{(1-\sigma_k)/2 -1 + \alpha \sigma_{k+1}} \ll 1$ by our definition of $\sigma_{k+1}$. Since $\alpha<1$ we find after a finite number of steps that
$\sigma_k \geq 1$, and then we can stop the iteration.


It remains to show that \eqref{ublargex3} implies the pointwise estimate for $f$. Indeed, \eqref{ublargex3} implies for $R>0$ that
\[
 R^{\gamma} \int_{R}^{2R} f(x)\,dx \leq \int_{R}^{2R} x^{\gamma} f(x)\,dx \leq \gamma^{\gamma} e^{A\gamma}
\]
and thus
\[
 \int_R^{2R} f(x)\,dx \leq \exp\Big( \gamma (\log(\gamma)+ \log(R)) + A \gamma\Big)\,.
\]
The minimum of $\psi(\gamma):= \gamma (\log(\gamma) +\log(R)) + A \gamma$ is obtained for $\gamma = e^{-(A+1)} R$ 
and thus
\[
 \int_R^{2R} f(x)\,dx \leq \exp\Big( - e^{-(A+1)} R\Big)
\]
and obviously it follows  that there exists (another) $A>0$ such that
\begin{equation}\label{ublargex20}
 \int_R^{2R} f(x)\,dx \leq \exp\Big( - A R\Big)\,.
\end{equation}
Notice that \eqref{ublargex20} also implies that
$ \int_R^{\infty} f(x)\,dx \leq \exp\Big( - A R\Big)$
for some $A>0$.

To deduce now the pointwise bound we first get from \eqref{eq1b} that
\[
\begin{split}
 x^2 f(x) &=\int_0^x\,dy \int_{x-y}^{\infty} \,dz\cdots \\
 &\leq \int_0^1\,dy\int_{x/2}^{\infty} \,dz
 \cdots + \int_0^1\,dz\int_{x-z}^x\,dy \cdots + \int_1^x\,dy \int_{\max(1,x-y)}^{\infty}\,dz\cdots\\
 &=:(I)+(II)+(III)\,.
 \end{split}
\]
Next, we argue via a dyadic argument (as in the proof of Lemma \ref{L.negativemoment}), using \eqref{ublargex20},
that $\int_{x/2}^{\infty} z^{\alpha}f(z)\,dz\leq Ce^{-\frac{A}{2}x}$.
With this estimate and \eqref{ubsmallx} we obtain that
\[
 |(I)| \leq C \int_0^1\int_{x/2} \Big(\frac{z}{y}\Big)^{\alpha} y f(y)f(z)\,dz\,dy \leq C e^{-\frac{A}{2}x}\,.
\]
Furthermore, given $x-y$ we choose $M$ such that $2^{-M}\leq x-y \leq 2^{-M+1}$ and obtain with \eqref{ubsmallx1} that
\[
 \int_{x-y}^1 z^{-\alpha} f(z)\,dz = \sum_{n=0}^M \int_{2^{-(n{+}1)}}^{2^{-n}} z^{-\alpha} f(z)\,dz \leq C2^{M\alpha}
 \leq C(x-y)^{-\alpha}\,.
\]
Hence
\[
\begin{split}
 |(II)| &\leq C \int_0^1 \int_{x-z}^x\Big(\frac{y}{z}\Big)^{\alpha}y f(y) f(z)\,dy\,dz
\\ & = C\int_{x-1}^x y f(y) \int_{x-y}^1z^{-\alpha}f(z)\,dz\,dy \leq  x \int_{x-1}^x (x-y)^{-\alpha} f(y)\,dy\,.
\end{split}
\]
Finally, we can estimate term $(III)$, using as above dyadic arguments and \eqref{ublargex20} as well as
\eqref{ubsmallx}, to obtain
\begin{align*}
|(III)| &= \int_1^x\int_{\max(1,x-y)}^\infty K(y,z)y f(y) f(z)\,dz\,dy\\
&=\int_{y \leq z}\cdots + \int_{y>z} \cdots\\
&\leq C\int_1^x\int_{\max(1,x-y)}^{\infty}\Big( \frac{z}{y}\Big)^{\alpha}yf(y) f(z) \,dz\,dy
+C\int_1^x\int_{\max(1,x-y)}^{\infty}\Big( \frac{y}{z}\Big)^{\alpha}yf(y) f(z) \,dz\,dy\\
& \leq C \int_{x/2}^{\infty} z^{\alpha} f(z)\,dz 
+ \int_{x/2}^{\infty} 
y^{1+\alpha}f(y)\,dy\\
& \leq C e^{-\frac{A}{2} x}\,.
\end{align*}
In summary we obtain
\begin{equation}\label{ublargex21}
 f(x)\leq Ce^{-Ax}+\frac{C}{x} \int_{x-1}^x\frac{f(y)}{(x-y)^{\alpha}} \,dy
\end{equation}
with a (possibly new) constant $A>0$. Iterating \eqref{ublargex21} gives, using Fubini,
\[
 f(x) \leq Ce^{-Ax} +C^2 \int_{x-1}^x\frac{e^{-Ay}}{(x-y)^{\alpha}}\,dy + C^2 \int_{x-2}^x f(y_2) \int_{y_2}^{y_2+1}
 \frac{1}{(y_1-y_2)^{\alpha}(x-y_1)^{\alpha}}\,dy_1\,dy_2\,.
\]
Now
\[
 \int_{x-1}^x \frac{e^{-Ay}}{(x-y)^{\alpha}}\,dy\leq B e^{-Ax}
\]
and
\[
  \int_{y_2}^{y_2+1}
 \frac{1}{(y_1-y_2)^{\alpha}(x-y_1)^{\alpha}}\,dy_1 \leq \frac{C}{(x-y_2)^{2\alpha-1}} \,.
\]
Hence, if $2\alpha<1$ we can finish the proof using Gronwall's inequality. If $2\alpha>1$ we iterate \eqref{ublargex21}
$L$ times to find
\begin{align*}
 f(x)&\leq C^{L+1}B^{L+1} e^{-Ax}\\
 &+ C^{L+2} \int_{x-(L+1)}^x f(y_{L+1}) 
 \int_{y_{L+1}}^{y_{L+1}+1} \frac{dy_L}{(y_l-y_{L+1})^{\alpha}} \cdots \frac{C}{(x-y_2)^{2\alpha-1}}dy_{L+1}\\
 & \leq C^{L+1}B^{L+1} e^{-Ax} + C^{L+2} \int_{x-(L+1)}^x f(y_{L+1}) \frac{C}{(x-y_2)^{(L+1)\alpha-L}}\,dy_{L+1}\,.
\end{align*}
If $L$ satisfies $(L+1)\alpha < L$ then we can also use Gronwall's inequality to conclude the proof. 
\end{proof}

\subsection{A lower bound}

\begin{lemma}\label{L.lowerbound}
 There exists $c,a>0$ such that
\[
 f(x)\geq c e^{-ax} \mbox{ for all } x\geq 1\,.
\]
\end{lemma}

\begin{proof}
We start with a lower bound on the integral $I(R):=\int_R^{R+1} f(x)\,dx$. From equation \eqref{eq1b} we deduce for large $R$ that
\begin{equation}\label{lowerbound1}
 \begin{split}
  I(R) & \geq \frac{k_0}{R-1}\int_{R}^{R+1} \int_{\frac{(1{-}\kappa)x}{2}}^{\frac{1{+}\kappa}{2}x} f(y)\int_{\max(x-y,\frac{1{-}\kappa}{1{+}\kappa}y)}^{
\frac{1{+}\kappa}{1{-}\kappa}y} f(z)\,dz\,dy\,dx\\
& \geq \frac{Ck_0}{R} \Big[ I\Big( \frac{R+1}{2}\Big)^2 + I \Big( \frac{R+1}{2}+1\Big) I\Big( \frac{R+1}{2}-1\Big) + \cdots \\
& \qquad \quad + I\Big(\frac{R+1}{2} + \Big \lfloor \frac{\kappa R}{4}\Big \rfloor \Big)  I\Big(\frac{R+1}{2} - \Big \lfloor \frac{\kappa R}{4}\Big \rfloor \Big)\Big]\,.
 \end{split}
\end{equation}
We can deduce from Lemma \ref{L.auxiliary} below that there exists $R_0>0$ such that $I(R) > \frac{8 e^a}{Ck_0\kappa} e^{-aR}$ for all $R \in [R_0,3R_0]$ for a sufficiently
large $a>0$.
Then \eqref{lowerbound1} implies with $K=\frac{8 e^a}{Ck_0\kappa}$ that
\[
 I(R) \geq \frac{Ck_0}{R} \Big \lfloor \frac{\kappa R}{2}\Big\rfloor^2 K^2 e^{-aR} > K e^{-aR)}\qquad \mbox{ for all } R \in \Big[ \frac{2R_0-1}{1-\kappa/2},\frac{2(3R_0)-1}{1+\kappa/2}\Big]\,.
\]
Choosing $\kappa$ smaller if necessary, we can assume that t $\frac{2R_0-1}{1-\kappa/2} \leq 3R_0$ and $\frac{2R-1}{1+\kappa/2}\geq \frac{3}{2} R$ for all $R \geq R_0$. This means that we have the estimate
in $[R_0,\frac{3}{2} 3 R_0]$. In the next step, we obtain the estimate in $[R_0,\frac{3^2}{2^2} 3 R_0]$ and iterating this argument we obtain
\begin{equation}\label{lowerbound2}
 I(R) \geq Ke^{-aR} \qquad \mbox{ for all } R \geq R_0\,.
\end{equation}
The pointwise lower bound now follows from using \eqref{lowerbound2} in \eqref{eq1b}, first for all $x \geq R_0$, and then by adapting the constants we obtain it for all $x \geq 1$.
\end{proof}

\begin{lemma}\label{L.auxiliary}
 There exists $\tilde R>0$ such that $f(x)>0$ for all $x \geq \tilde R$.
\end{lemma}
\begin{proof}
 We first claim that there exists $x_0>0$ such that 
\begin{equation}\label{aux1}
 \int_{B_{\frac{\kappa x_0}{4}(x_0)}}x f(x)\,dx >0\,.
\end{equation}
This follows from the fact that
\[
 (0,\infty) \subset \cup_n B_{\frac{\kappa x_n}{4}(x_n)} \qquad \mbox{ with } \qquad x_n=\Big (1+ \frac{\kappa}{8}\Big)^n\,.
\]
Then
\[
 M = \int_0^{\infty} x f(x)\,dx \leq \sum_n \int_{B_{\frac{\kappa x_n}{4}(x_n)} }xf(x)\,dx 
\]
and hence there exists $n_0 \in \N$ such that \eqref{aux1} is true for $x_0:=x_{n_0}$.
Then, since 
\[
 B_{\frac{\kappa x_0}{4}}(x_0) \times  B_{\frac{\kappa x_0}{4}}(x_0) \subset \{ (y,z)\,:\, |y-z| \leq \kappa(y+z)\}\,,
\]
equation \eqref{eq1b} implies that
\[
 x^2 f(x) \geq k_0 \Big( \int_{B_{\frac{\kappa x_0}{4}(x_0)} }y f(y)\,dy \Big)^2 >0 \qquad \mbox{ for all } x \in 
\Big( x_0 \big( 1+ \frac{\kappa}{4}\big), 2x_0 \big(1-\frac{\kappa}{4}\big)\Big)\,.
\]
We can then iterate this argument to obtain $f(x)>0$ for all $x \geq x_0 \big( 1+ \frac{\kappa}{4}\big)$.

\end{proof}

\section{BV regularity}\label{S.bv}

\begin{lemma}\label{L.bv}
If $K$ satisfies \eqref{Kassump1}, \eqref{Kassump2} and \eqref{Kdiff}  if  $f$  is a solution to \eqref{eq2} and  \eqref{eq1b}, then
 $f \in BV_{\mbox{loc}}(0,\infty)$.
\end{lemma}
\begin{proof}
In the following we let $\phi \in C^1_b(0,\infty)$ with $\mbox{supp}(\phi) \subset (0,\infty]$ and $\mbox{supp}(\phi') \subset (0,\infty)$.  Our goal is to show that
\begin{equation}\label{bv2}
 \Big|\int_0^{\infty}  f(x) \phi'(x) \,dx\Big|\leq C \|\phi\|_{L^{\infty}(0,\infty)}\,.
\end{equation}
In a first step we take $\phi$ such that $\mbox{supp}(\phi)\subset [a,
\infty]$ and $\mbox{supp}(\phi') \subset [a,b]$ for some $a,b>0$. We divide \eqref{eq1b} by $x^2$, multiply with $\phi'$, integrate and change the order of integration 
to obtain
\begin{equation}\label{bv3}
\begin{split}
 \int_0^{\infty} &f(x)  \phi'(x)\,dx= \int_0^{\infty} \int_0^{\infty} K(y,z) y f(y) f(z) \Big( \int_y^{y{+}z} \frac{\phi'(x)}{x^2}\,dx\Big)\,dz\,dy\\
& = \int_0^{\infty}  \int_0^{\infty}  \,K(y,z) y f(y) f(z) \Big( 2 \int_y^{y{+}z} \frac{\phi(x)}{x^3}\,dx + \frac{\phi(y{+}z)}{(y{+}z)^2} - \frac{\phi(y)}{y^2}\Big)\,dz\,dy\,.
\end{split}
\end{equation}
We split $\int_0^{\infty} dz = \int_0^{a/2} dz \dots + \int_{a/2}^{\infty} dz\dots$. Using the assumption on the support of $\phi$ and recalling 
Lemma \ref{L.negativemoment} and Lemma \ref{L.ublargex} we can estimate 
\begin{equation}\label{bv4}
 \begin{split}
  \Big| \int_0^{\infty} \int_{a/2}^{\infty}  &\,K(y,z) y f(y) f(z) \Big( 2 \int_y^{y{+}z} \frac{\phi(x)}{x^3}\,dx + \frac{\phi(y{+}z)}{(y{+}z)^2} - \frac{\phi(y)}{y^2}\Big)\,dz\,dy\Big|\\
& \leq C \int_{a/2}^{\infty} y^{-2} f(y) \int_{a/2}^{\infty}\Big( \Big(\frac{y}{z}\Big)^{\alpha}+\Big(\frac{z}{y}\Big)^{\alpha}\Big) f(z) \,dz\,dy \|\phi\|_{L^{\infty}} \\
&\leq C  \|\phi\|_{L^{\infty}}\,.
 \end{split}
\end{equation}
Here and in the following the constants depend on $a$, but we will not write this explicitly. Next, we write
\begin{equation}\label{bv4b}
\begin{split}
 2 \int_y^{y{+}z} \frac{\phi(x)}{x^3}\,dx & + \frac{\phi(y{+}z)}{(y{+}z)^2} - \frac{\phi(y)}{y^2} = \\
& 
 2 \int_y^{y{+}z} \frac{\phi(x)}{x^3} \,dx + \frac{\phi(y{+}z)}{(y{+}z)^2} - \frac{\phi(y{+}z)}{y^2} + \frac{1}{y^2} \big( \phi(y{+}z)-\phi(y)\big)
\end{split}
\end{equation}
and estimate for $y \geq a/2$ that 
\[
 \Big | 2 \int_y^{y{+}z} \frac{\phi(x)}{x^3}\,dx + \frac{\phi(y{+}z)}{(y{+}z)^2} - \frac{\phi(y{+}z)}{y^2} \Big| \leq C \|\phi\|_{L^{\infty}}\, z\,. 
\]
With this last estimate we find, recalling again Lemmas \ref{L.negativemoment} and \ref{L.ublargex} that 
\begin{equation}\label{bv5}
 \begin{split}
   \Big| \int_0^{\infty}  \int_0^{a/2}  K(y,z) y f(y) f(z) &\Big(2 \int_y^{y{+}z} \frac{\phi(x)}{x^3} \,dx+ \frac{\phi(y{+}z)}{(y{+}z)^2} - \frac{\phi(y{+}z)}{(y{+}z)^2} \Big)\,dz\,dy\Big|\\
& \leq C  \|\phi\|_{L^{\infty}}\,.
 \end{split}
\end{equation}
Thus, with \eqref{bv4}, \eqref{bv4b} and \eqref{bv5}, we have obtained
\begin{equation}\label{bv6}
  \Big|\int_0^{\infty}  f(x) \phi'(x) \,dx\Big|\leq C \|\phi\|_{L^{\infty}} + \Big| \int_0^{\infty} \int_0^{a/2}  K(y,z) y f(y) f(z) \frac{1}{y^2} 
\big( \phi(y{+}z)-\phi(y)\big)\,dz\,dy\Big|
\end{equation}
and it remains to estimate the last term on the right hand side of \eqref{bv6}.

Our strategy is to write this term as $\int f(y) (T\phi)'(y)$ with
\[
 (T\phi)'(y) = \frac{1}{y} \int_0^{a/2} f(z) K(y,z) \big( \phi(y{+}z)-\phi(y)\big)\,dz
\]
and
\[
 (T \phi)(y) = \int_0^y  \int_0^{a/2}  f(z)K(\xi,z) \frac{1}{\xi} \big( \phi(\xi{+}z)-\phi(\xi)\big)\,dz\,d\xi
\]

and iterate the previous estimates for  the function $T\phi$. We first need to verify that  $T\phi$ is an admissible test function. We easily check that if
$\mbox{supp} (\phi) \subset [a,\infty]$ and $\mbox{supp} (\phi') \subset [a,b]$ that then $\mbox{supp} (T\phi) \subset [a/2,\infty]$ and $\mbox{supp}( (T \phi)') \subset [a/2,b]$ and
hence we can use $T\phi$ as test function. Estimate \eqref{bv6} then implies
\begin{equation}\label{bv7}
 \Big |\int_0^{\infty}  f(y) \big(T\phi\big)'(y)\,dy\Big| \leq C  \|T\phi\|_{L^{\infty}} + 
\Big| \int_0^{\infty}  f(y) \big( T^2\phi\big)'(y)\,dy\Big|\,.
\end{equation}

Thus, our first task is to estimate $\|T\phi\|_{L^{\infty}}$. Notice first that \eqref{Kdiff} implies
\begin{equation}\label{Kdiff1}
 \Big | \frac{K(x-y,y)}{x-y} - \frac{K(x,y)}{x}\Big| \leq C y^{1-\alpha} \qquad \mbox{ for } x \geq \frac{a}{2}\;,\; y \in (0,a/2]\,.
\end{equation}
Using \eqref{Kdiff1},  $\mbox{supp}(\phi) \subset [a,\infty)$ 
 and  Lemma \ref{L.negativemoment}, we find
\begin{equation}\label{bv8}
 \begin{split}
  (T\phi)(y)& = \int_0^{a/2} f(z) \int_{a/2}^{y}  \frac{K(\xi,z)}{\xi} \phi(\xi{+}z)\,d\xi\,dz - \int_0^{a/2}  f(z) \int_a^y  \frac{K(\xi,z)}{\xi} \phi(\xi)\,d\xi\,dz\\
&= \int_{0}^{a/2}  f(z) \Big [ \int_{a/2}^y  \frac{K(\xi,z)}{\xi} \phi(\xi+z)\,d\xi - \int_a^y  \frac{K(\xi,z)}{\xi} \phi(\xi)\,d\xi\Big]\,dz\\
&= \int_0^{a/2} f(z) \int_y^{y+z} \frac{K(\xi-z,z)}{\xi-z} \phi(\xi)\,d\xi\,dz \\
&\quad + \int_0^{a/2}  f(z) \int_{a/2+z}^{y}\  \Big(
\frac{K(\xi-z,z)}{\xi-z} - \frac{K(\xi,z)}{\xi}\Big)\phi(\xi)\,d\xi\,dz \\ 
& \leq C \|\phi\|_{L^{\infty}} \int_0^{a/2}  f(z) z^{1-\alpha} \,dz+ C \int_0^{a/2} f(z) \int_{a/2+z}^y   z^{1-\alpha} \phi(\xi)\,d\xi\,dz\\
& \leq C \|\phi\|_{L^{\infty}}\,.
\end{split}
\end{equation}
It remains to estimate $\Big| \int_0^{\infty}  f(y) \big( T^2\phi\big)'(y)\,dy\Big|$. To that aim recall
\begin{equation}\label{bv9}
 \big( T^2\phi\big)'(y) = \frac{1}{y} \int_0^{a/2} f(z)K(y,z) \Big( \big(T\phi\big)(y{+}z) - \big(T\phi\big)(y) \Big)\,dz
\end{equation}
and we write
\begin{equation}\label{bv10}
 \begin{split}
(T\phi)(y+z)&-(T\phi)(y) = \int_y^{y+z} \int_0^{a/2}  f(\eta) K(\xi,\eta) \frac{1}{\xi}\big( \phi(\xi+\eta)-\phi(\xi)\big)\,d\eta\,d\xi\\
&=\int_0^{a/2} f(\eta) \Big( \int_{y+\eta}^{y+\eta+z} \frac{K(\xi-\eta,\eta)}{\xi-\eta} \phi(\xi)\,d\xi - 
\int_y^{y+z}\frac{K(\xi,\eta)}{\xi} \phi(\xi)\,d\xi\Big)\,d\eta\\
&= \int_0^{a/2}  f(\eta)  \int_{y+\eta}^{y+\eta+z} W(\xi,\eta)\phi(\xi)\,d\xi\,d\eta\\
&\quad  
+ \int_0^{a/2}  f(\eta) \Big(\int_{y+z}^{y+z+\eta}  \frac{K(\xi,\eta)}{\xi}\phi(\xi) \,d\xi- \int_y^{y+\eta}  \frac{K(\xi,\eta)}{\xi}\phi(\xi)\,d\xi\Big)\,d\eta\\
&=: (T_1\phi)(y+z)-(T_1\phi)(y) + (T_2\phi)(y+z)-(T_2\phi)(y)\,,
 \end{split}
\end{equation}
where
\[
 W(\xi,\zeta)= \frac{K(\xi- \zeta,\xi)}{\xi-\zeta} - \frac{K(\xi,\zeta)}{\xi}\,.
\]
Invoking again \eqref{Kdiff1} we find
\begin{equation}\label{bv13}
 \begin{split}
\Big|  (T_1\phi)(y{+}z)-(T_1\phi)(y)\Big| 
&= \Big| \int_0^{a/2} f(\zeta) \int_y^{y{+}z} W(\xi,\zeta) \phi(\xi)\,d\xi\,d\zeta\Big|\\
&\leq C \|\phi\|_{L^{\infty}} z \int_0^{a/2}  f(\zeta) \zeta^{1-\alpha} \,d\zeta \leq C \|\phi\|_{L^{\infty}}z\,.
\end{split}
\end{equation}
If we could obtain the same estimate for $T_2$ we would obtain from \eqref{bv13} together with  Lemma \ref{L.negativemoment} 
and the properties of  $\mbox{supp} (\phi)$ that
$\|T^2\phi\|_{0,1} \leq C \|\phi\|_{L^{\infty}}$ and the proof would be finished.

Unfortunately, we cannot in general  expect an estimate for  $T_2\phi$ as in \eqref{bv13}. The reason is that,  using \eqref{Kassump1}, we obtain 
\[
\big|(T_2\phi)(y+z)-(T_2\phi)(y)\big|
 \leq C_a \|\phi\|_{L^{\infty}} z^{1-\alpha} \int_0^{a/2}f(\eta) \,d\eta \,,
\]
but in general we do not know whether the integral $\int_0^{a/2} \,d\eta f(\eta)$ is finite and in addition the factor $z^{1-\alpha}$ causes problems in \eqref{bv9}
if $\alpha \geq 1/2$.

Thus, in order to obtain the desired estimates we need to iterate again. More precisely, we define $L \in \N$ such that $L>1+\frac{1}{1-(\alpha+\eps)}$ with
some fixed $\eps \in (0,(1-\alpha)/2)$.
We are going to show that
\begin{equation}\label{bv15}
 \|T^L\phi\|_{0,1} \leq C \|\phi\|_{L^{\infty}}\,.
\end{equation}
We have seen that $\mbox{supp}(\phi) \subseteq [a,\infty]$ implies $\mbox{supp}(T^L\phi) \subseteq [\frac{a}{2^L},\infty]$, whereas
$\mbox{supp}(\phi') \subseteq [a,b]$ implies $\mbox{supp}((T^L\phi)') \subseteq [\frac{a}{2^L},b]$. Hence, $T^L\phi$ is an admissible test function and
we obtain from \eqref{bv7}, \eqref{bv8} and \eqref{bv15} that
\[
 \Big| \int_0^{\infty}  f(y) (T^L\phi)'(y)\,dy\Big| \leq C \|T^L\phi\|_{L^{\infty}} + \Big| \int_0^{\infty} f(y) (T^L\phi)'(y)\,dy\Big|
\leq C \|\phi\|_{L^{\infty}}\,,
\]
which finishes the proof of Lemma \ref{L.bv}.

Thus, it remains to prove \eqref{bv15}. We split the integral in the definition of $T_2$ as
$\int_0^{a/2}\,d\eta \dots = \int_0^z \,d\eta \cdots + \int_z^{a/2} \,d\eta \dots$ and write
\begin{equation}\label{bv14}
 \begin{split}
H(y&,z,\eta):= \int_{y+z}^{y+z+\eta} \frac{K(\xi,\eta)}{\xi} \phi(\xi) \,d\xi- \int_y^{y+\eta}  \frac{K(\xi,\eta)}{\xi}\phi(\xi)\,d\xi\\
&=
\int_y^{y+\eta}  \Big( \frac{K(\xi+z,\eta)}{\xi+z} - \frac{K(\xi,\eta)}{\xi}\Big) \phi(\xi+z)\,d\xi + \int_y^{y+\eta} \frac{K(\xi,\eta)}{\xi} \big(\phi(\xi+z)-\phi(\xi)\big)\,d\xi\\
&=:(I)+(II)\,.
 \end{split}
\end{equation}
Assumption \eqref{Kdiff} implies that
\[
 |(I)| \leq C_a \|\phi\|_{L^{\infty}} \eta^{1-\alpha} z\,,
\]
while \eqref{Kassump1} gives
\[
|(II)| \leq C_a  \|\phi\|_{0,\gamma} \eta^{1-\alpha} z^{\gamma} 
\]
for $\gamma \in [0,1)$ such that
\begin{equation}\label{bv16}
\Big| \int_0^z  f(\eta) H(y,z,\eta) \,d\eta\Big| \leq C\Big( \|\phi\|_{L^{\infty}} z \int_0^z  \eta^{1-\alpha} f(\eta)\,d\eta  
+ \|\phi\|_{0,\gamma} z^{\gamma} \int_0^z \eta^{1-\alpha} f(\eta)\,d\eta\,.
\end{equation}
A dyadic argument, as in the proof of Lemma \ref{L.negativemoment}, implies that $\int_0^z \eta^{1-\alpha} f(\eta) \,d\eta\leq C z^{1-\alpha}$ 
such that
\begin{equation}\label{bv17}
\Big| \int_0^z  f(\eta) H(y,z,\eta)\,d\eta \Big| \leq C\Big( \|\phi\|_{L^{\infty}} z^{2-\alpha} + \|\phi\|_{0,\gamma} z^{1-\alpha+\gamma}\Big) \,.
\end{equation}
Next, we notice that in the case that $\eta >z$ it holds
\[
 H(y,z,\eta)= \int_{y+\eta}^{y+\eta+z} \,d\xi \dots + \int_y^{y+z} \,d\xi \dots\,.
\]
Hence we can derive, using \eqref{Kdiff1}, that 
\begin{equation}\label{bv18}
 \begin{split}
\Big| \int_{y+\eta}^{y+\eta+z}  & \frac{K(\xi,\eta)}{\xi} \phi(\xi)\,d\xi - \int_y^{y+z} \frac{K(\xi,\eta)}{\xi} \phi(\xi)\,d\xi\Big|
\\ & = \Big| \int_y^{y+z} \Big[ \Big( \frac{K(\xi+\eta,\eta)}{\xi+\eta} - \frac{K(\xi,\eta)}{\xi} \Big) \phi(\xi+\eta)  +
\frac{K(\xi,\eta)}{\xi} \big( \phi(\xi+\eta) - \phi(\xi)\big)\Big]\,d\xi \Big|\\
& \leq C \|\phi\|_{L^{\infty}} z \eta^{1-\alpha} + C \|\phi\|_{0,\gamma} z \eta^{\gamma-\alpha}\,.
 \end{split}
\end{equation}
Using estimate \eqref{bv18} we obtain
\begin{equation}\label{bv19}
 \Big|\int_z^{a/2} f(\eta) H(y,z,\eta)\,d\eta\Big| \leq C \|\phi\|_{L^{\infty}} z + C \|\phi\|_{0,\gamma} z \int_z^{a/2}  \eta^{\gamma-\alpha} f(\eta)\,d\eta\,.
\end{equation}
We assume that $\gamma \leq \alpha$, since eventually we want to use the estimate also for $\gamma =0$. We obtain with Lemma \ref{L.negativemoment}
\begin{equation}\label{bv20}
 \int_z^{a/2}  \eta^{\gamma-\alpha} f(\eta) \,d\eta\leq C_{\eps} z^{\gamma-\alpha - \eps}\,.
\end{equation}
 In summary, \eqref{bv10}, \eqref{bv17}, \eqref{bv19} and  \eqref{bv20} imply
\[
 \Big| (T\phi)(y+z) - (T\phi)(y)\Big| \leq C\big( \|\phi\|_{L^{\infty}} z + \|\phi\|_{0,\gamma} z^{1-\alpha-\eps + \gamma} \big)
\]
and consequently
\begin{equation}\label{bv21}
 \|T\phi\|_{0,1-\alpha-\eps+\gamma} \leq C \|\phi\|_{0,\gamma}\,.
\end{equation}
We can iterate \eqref{bv21} to obtain
\begin{equation}\label{bv22}
 \|T^L\phi\|_{0,1-\alpha-\eps+\gamma_L} \leq C \|T^{L-1}\phi\|_{0,\gamma_L} 
=C \|T^{L-1}\phi\|_{0,1-\alpha-\eps +\gamma_{L-1}} \leq \dots \leq C \|\phi\|_{0,\gamma_1}\,.
\end{equation}
With the definitions $\gamma_\ell:=\min(1,(\ell-1)(1-\alpha-\eps))$, $\ell=1,\dots,L$ 
and the choice of $L$ we have $\gamma_L \geq 1$ and $\gamma_1=0$ which finishes the proof of \eqref{bv15}.
\end{proof}


\section{Asymptotics as $x \to \infty$}\label{S.asymptotics}

\subsection{A lower bound on the changes of $\log \frac{1}{f(x)}$}
Writing 
\[
 f(x)=e^{-x a(x)}
\]
our first goal is to derive a lower bound on $a'(x)$. Notice that since $f \in C^0 \cap BV_{loc}$ and $f>0$, we also have that $a \in BV_{loc}(0,\infty)$ (see e.g. \cite{AmDalMa90}).

Our first goal will be to derive a lower bound on $a'$. For that purpose we need the following Lemma.
\begin{lemma}\label{L.lowerasymp1}
Suppose that $a\colon \R\to \R$ such that
\begin{equation}\label{lowerasymp1a}
 0<\alpha_1 \leq a(x) \leq \alpha_2< \infty \qquad \mbox{ for all } x \geq 1\,.
\end{equation}
Then, given an $\eps\in(0,1)$ and $R\geq 1$ there exists $\bar x  \in I_R:=\Big( R\big(1+\frac{\eps}{4}\big), R\big(1+\frac{3\eps}{4}\big)\Big)$ such that
\[
 a(y) \leq a(\bar x) + \frac{C(\eps)}{\bar x} \qquad \mbox{ for all } y \in [\bar x-1,\bar x] \cap I_R
\]
 with $C(\eps)= \frac{4(\alpha_2-\alpha_1)}{\eps}$.
\end{lemma}

\begin{proof} 
Consider first the case that $R \eps \geq 2$. Then we have
\[
 a(y) \leq \alpha_2 = \alpha_1+\frac{C(\eps)\eps}{4} \leq \alpha_1 + \frac{C(\eps)}{2R} \leq \alpha_1 + \frac{C(\eps)}{2} 
 \frac{\big(1+\frac{3\eps}{4}\big)}{\bar x}
\leq a(\bar x) + \frac{C(\eps)}{\bar x}
\]
for all $\bar x \leq R\big( 1+\frac{3\eps}{4}\big)$.

Now assume that $R\eps >2$. Suppose there exists $R \geq \frac{2}{\eps}$ such that for all $\bar x \in I_R$ there exists
$y \in [\bar x-1,\bar x]\cap I_R$ with $a(y)>a(\bar x) + \frac{C(\eps)}{\bar x}$. We define a sequence $(x_n)$ as follows: we set $x_0=R(1+\frac{3\eps}{4})$.
By assumption there exists $x_1 \in [x_0-1,x_0]\cap I_R$ such that $a(x_1)>a(x_0)+ \frac{C}{x_0}$. Iteratively we obtain $x_n \in [x_{n-1}-1,x_{n-1}]\cap I_R$ 
such that $a(x_n) \geq a(x_0) + C \big( \frac{1}{x_0} + \dots + \frac{1}{x_{n-1}}\big) \geq a(x_0) + \frac{Cn}{x_0}$. For $n=\lfloor \frac{R\eps}{2}\rfloor$ we obtain
\[
a(x_n) \geq \alpha_1 + \frac{4 (\alpha_2-\alpha_1)}{\eps} \frac{R\eps}{2 R\big( 1+\frac{3\eps}{4}\big)}
=\alpha_1 + \frac{2 (\alpha_2-\alpha_1)}{ 1+ \frac{3\eps}{4}} > \alpha_2
\]
which gives a contradiction and finishes the proof of the Lemma.
\end{proof}

\begin{lemma}\label{L.lowerasymp2}
There exists $R>0$ and $b>0$  such that 
\[
 \int_{x}^{x+1} (a')_{-}(x)\,dx \leq \frac{C}{x}+ C_R e^{-b x}\qquad \mbox{ for all } x \geq R\,.
\]
Here we use the Jordan decomposition $a'(x)=(a')_{+}(x)-(a')_{-}(x)$ 
with $(a')_{+}\geq 0$ and $(a')_{-}\geq 0$. 
\end{lemma}

\begin{proof}
In order to make the idea of the proof clear, we first present the formal derivation of the result, that is we ignore for the moment that 
some of the integrals are not well-defined.
From \eqref{eq1} we obtain 
\begin{equation}\label{lowerasymp1}
\begin{split}
 f'(x)& = - \frac{2}{x}f(x) + \frac{1}{x}  \int_0^1 \big( K(x,y) f(x) - K(x{-}y,y) f(x{-}y) \big) f(y) \,dy
\\
& \quad + \frac{1}{x} \Big( \int_1^{\infty}  K(x,y) f(x)f(y) \,dy - \int_1^{x/2}  K(x{-}y,y) f(x{-}y)f(y)\,dy\Big)\,.
\end{split}
\end{equation}
Using  Lemma \ref{L.ublargex} to conclude that $\frac{1}{x}\int_1^{\infty} dy\, K(x,y) f(y)\leq Cx^{\alpha-1}$,
 we find
\begin{equation}\label{lowerasymp2}
 f'(x) \leq \frac{C}{x^{1-\alpha}} f(x) + \frac{f(x)}{x} \int_0^1 K(x{-}y,y) f(y) \Big(\frac{ K(x,y)}{K(x{-}y,y)}  -  \frac{f(x{-}y)}{f(x)}\Big) f(y)\,dy\,.
\end{equation}
Using \eqref{Kdiff} and Lemma \ref{L.negativemoment} we can estimate for $x \geq 1$ that
\[
\begin{split}
\Big | \frac{f(x)}{x} \int_0^1 K(x{-}y,y) f(y) \Big(  \frac{ K(x,y)}{K(x{-}y,y)} -1\Big)\,dy \Big| 
& \leq \frac{f(x)}{x} C \int_0^1 \Big( \frac{x}{y}\Big)^{\alpha} \frac{y}{x} f(y)\,dy \\
&\leq \frac{f(x)}{x^{2-\alpha}} \leq \frac{f(x)}{x^{1-\alpha}}
\end{split}
\]
and thus we can absorb this error term into the first term of the right hand side of \eqref{lowerasymp2}.
In terms of $a$ we have found the inequality
\begin{equation}\label{lowerasymp3}
 \begin{split}
  \big( x a(x)\big)' & \geq - \frac{C}{x^{1-\alpha}} + \frac{1}{x} \int_0^1  K(x{-}y,y)f(y)  \Big( e^{-(x{-}y) a(x{-}y) + xa(x)} -1\Big) \,dy\\
& = - \frac{C}{x^{1-\alpha}} + \frac{1}{x} \int_0^1  K(x{-}y,y) f(y)\Big( e^{ya(x{-}y)} e^{x(a(x)-a(x{-}y))} -1\Big) \,dy \\
& \geq - \frac{C}{x^{1-\alpha}} + \frac{1}{x} \int_0^1 K(x{-}y,y)f(y)\Big(e^{x(a(x)-a(x{-}y))} -1\Big) \,dy\,.
 \end{split}
\end{equation}
Estimate \eqref{lowerasymp3} implies  that
\begin{equation}\label{lowerasymp4}
 x (a')_{+}(x)-x (a')_{-}(x)+a(x)\geq - \frac{C}{x^{1-\alpha}}+ \frac{1}{x}  \int_0^1  K(x{-}y,y)f(y)\Big(e^{x(a(x)-a(x{-}y))} -1\Big)\,dy\,.
\end{equation}
Furthermore, since $a$ is continuous, we can write and estimate
\[
 a(x)-a(x-y) = \int_{x{-}y}^x a'(\xi)\,d\xi \geq - \int_{x{-}y}^x (a')_{-}(\xi) \,d\xi \,.
\]
Thus we have, using $1-e^{-z} \leq z$ for $z \geq 0$ and a dyadic argument as in Lemma \ref{L.negativemoment}, that
 \begin{equation}\label{lowerasymp6}
\begin{split}
 x (a')_{-}(x) &\leq \alpha_2 +\frac{C}{x^{1-\alpha}} + \frac{1}{x} \int_0^1  K(x{-}y,y) f(y)\Big(1- e^{-x \int_{x{-}y}^x (a')_{-}(\xi) \,d\xi } \Big) \,dy\\
& \leq C + C \int_0^1  \int_{x{-}y}^x (a'(\xi))_{-}\,d\xi \frac{(x-y)^{\alpha}}{y^{\alpha}} f(y)\,dy\\
&\leq C + C \int_{x{-}1}^x (a')_{-}(\xi) \int_{x{-}\xi}^1 \frac{f(y)}{y^{\alpha}} (x-y)^{\alpha} \,dy \,d\xi\\
& \leq C + C x^{\alpha} \int_{x{-}1}^x (a')_{-}(\xi)\frac{1}{(x-\xi)^{\alpha}}\,d\xi\,.
\end{split}
\end{equation}
As a consequence, we obtain
\[
 (a')_-(x) \leq \frac{C}{x} + \frac{C}{x^{1-\alpha}} \int_{x-1}^x (a')_-(\xi) \frac{1}{(x-\xi)^{\alpha}}\,d\xi
\]
 and with $F(x):= \int_x^{x+1} (a')_-(y)\,dy$ it follows 
\begin{equation}\label{lowerasymp7}
\begin{split}
F(x)& \leq \frac{C}{x} + \frac{C}{x^{1-\alpha}} \int_{x}^{x+1} \int_{\xi-1}^{\xi} \frac{(a')_-(y)}{(\xi-y)^{\alpha}}\,dy\,d\xi\\
& \leq \frac{C}{x} + \frac{C}{x^{1-\alpha}} \int_{x-1}^{x+1} (a')_-(y) \int_{\max(x,y)}^{\max(x+1,y+1)} \frac{1}{(\xi-y)^{\alpha}}\,d\xi \,dy\\
& \leq \frac{C}{x} + \frac{C}{x^{1-\alpha}} F(x-1) + \frac{C}{x^{1-\alpha}} F(x)\,.
\end{split}
\end{equation}
If $x \geq R$ where $C/R^{1-\alpha} \leq \frac 1 2$ we have found that
\begin{equation}\label{lowerasymp8}
 F(x) \leq \frac{C}{x} + \frac 1 2 F(x-1) \qquad \mbox{ for all } x \geq R\,.
\end{equation}
Since $a \in BV_{loc}$ we can assume that $F(R-1)<\infty$. Hence iterating \eqref{lowerasymp2} gives the desired result.

In order to derive \eqref{lowerasymp7} rigorously we have to work with the weak formulation of \eqref{lowerasymp1}. Multiplying with a
test function $\phi \in C_0^1(R,\infty)$ with sufficiently large $R\geq 1$, we find
\begin{equation}\label{rigorous1}
 - \int_0^{\infty} x^2 f(x)\phi'(x)\,dx = \int_0^{\infty} \int_0^{\infty} K(y,z) y \big ( \phi(y)-\phi(y+z)\big) f(y) f(z)\,dy\,dz\,.
\end{equation}
From \cite{AmDalMa90} we know that 
\begin{equation}\label{rigorous2}
- \int_0^{\infty} x^2 f(x)\phi'(x)\,dx = - \int_0^{\infty} x^2 \big( a(x) + x a'(x)\big) f(x)\phi(x)\,dx\,.
\end{equation}
We are going to show that 
\begin{equation}\label{rigorous3}
 - \int_0^{\infty} x^2 f(x)\phi'(x)\,dx \leq C \int_0^{\infty} x^{2{+}\alpha} \phi(x)f(x)\,dx + C \int_0^{\infty}
 \big( a'\big)_-(\xi)\int_{\xi}^{\xi+1} \frac{x^{2+\alpha} f(x)\phi(x)}{(x-\xi)^{\alpha}}\,dx\,d\xi\,.
\end{equation}
Plugging \eqref{rigorous2} into \eqref{rigorous3} and absorbing the term $\int x^2 a(x) f(x)\phi(x)\,dx$ into the first term on the right hand side
we obtain
\begin{equation}\label{rigorous4}
\begin{split}
 -\int_0^{\infty} x^3 a'(x) f(x)\phi(x)\,dx 
 &\leq C \int_0^{\infty} x^{2+\alpha} \phi(x)f(x)\,dx\\
& + C \int_0^{\infty}
 \big( a'\big)_-(\xi)\int_{\xi}^{\xi+1} \frac{x^{2+\alpha} f(x)\phi(x)}{(x-\xi)^{\alpha}}\,dx\,d\xi\,.
\end{split}
\end{equation}
Given a set $A \subset \R$ we take (after approximation) $\phi = \frac{\chi_A(x)}{f(x)}$, where the support of $A$ is such that $f>0$ on $A$
(cf. Lemma \ref{L.auxiliary}).
Defining 
\[
 \nu(A):= \int_0^{\infty} \big( a'\big)_-(\xi)\int_{\xi}^{\xi+1} \frac{x^{2+\alpha} \chi_A(x)}{(x-\xi)^{\alpha}}\,dx\,d\xi
\]
we obtain
\begin{equation}\label{rigorous5}
 -x^3 \big( a'\big)(A)\leq C\int_{A} x^{2+\alpha} \,dx + \nu(A)\,.
\end{equation}
We also know that given $A$ there exists $A_{\pm}$ such that
$\big(a'\big)_+(A_-)=0$, $\big(a'\big)_-(A_+)=0$ and 
$\big(a'\big)_+(A \cap A_+)= \big(a'\big)_+(A)$, $\big(a'\big)_-(A\cap A_-)=\big(a'\big)_-(A)$. Then we deduce from \eqref{rigorous5} that
\begin{equation}\label{rigorous6}
 x^3 \big(a'\big)_-(A) \leq C \int_A x^{2+\alpha}\,dx + \nu(A)
\end{equation}
and by choosing $A=[\bar x,\bar x+1]$ we obtain 
\[
 \begin{split}
\int_{\bar x}^{\bar x+1} \big(a'\big)_-(\xi)\,d\xi & \leq \frac{C}{\bar x} + \frac{C}{{\bar x}^3} \int_0^{\infty}
\big( a'\big)_- (\xi) \int_{\max(\xi,\bar x)}^{\max(\xi+1,\bar x+1)} \frac{x^{2+\alpha}}{(x-\xi)^{\alpha}}\,dx\,d\xi\\
& \leq \frac{C}{\bar x} + \frac{C}{{\bar x}^{1-\alpha}} \int_{[\bar x-1,\bar x+1]} \big(a'\big)_-(\xi)\,d\xi\,, 
 \end{split}
\]
which is just 
 \eqref{lowerasymp7}.

It remains to prove \eqref{rigorous3}. To that aim we rewrite the right hand side of \eqref{rigorous1} as
\[
 \int_0^{\infty} \,dy \int_1^{\infty} \,dz \dots +\int_0^{\infty} \,dy \int_0^1\,dz \dots =: (I)+(II)\,.
\]
Due to the properties of $\phi$ and $K$ we have
\[
 |(I)| \leq C \int_0^{\infty}y^{1+\alpha} f(y)\phi(y)\,dy \leq C \int_0^{\infty}y^2f(y)\phi(y)\,dy\,.
\]
We rewrite the second term as
\[
 \begin{split}
  (II)&=\int_0^{\infty}\int_0^1 f(y)f(z) \big[y K(y,z) \phi(y) -(y-z)K(y-z,z)\phi(y)\big]\,dz\,dy\\
  & \quad + \int_0^{\infty} \int_0^1 f(y)f(z) \big[(y-z)K(y-z,z)\phi(y) - yK(y,z)\phi(y+z)\big]\,dz\,dy\\
  &=:(II)_a+(II)_b\,.
 \end{split}
\]
Due to \eqref{Kdiff}, \eqref{ubsmallx} and the properties of $\phi$ we can  estimate 
\[
 \big|(II)_a\big|  \leq C \int_0^{\infty} f(y) \phi(y) y^{\alpha} \int_0^1 f(z) z^{1-\alpha}\,dz\,dy \leq  
C \int_0^{\infty} y^2 f(y) \phi(y)\,dy\,.
\]
As before in the formal argument we have the estimate
\begin{equation}\label{rigorous7}
\frac{f(y)}{f(y+z)} = e^{z a(z)} e^{y(a(y+z)-a(y))} \geq \exp\Big(-y \int_y^{y+z} \big(a'\big)_-(\xi)\,d\xi\Big)
\geq 1-y \int_y^{y+z} \big(a'\big)_-(\xi)\,d\xi\,.
\end{equation}
After a change of variables the term $(II)_b$ can be estimated as  
\[
 \begin{split}
 (II)_b&= \int_0^{\infty} \int_0^1 f(z) \big( f(y+z)-f(y)\big) y   K(y,z) \phi(y+z)\,dz\,dy\\
& \leq \int_0^{\infty} \int_0^1 f(z) f(y+z) y^2   K(y,z) \phi(y+z)\int_y^{y+z} \big( a'\big)_-(\xi)\,d\xi\,dz\,dy\\
&= \int_0^{\infty} \int_0^1 f(z) f(y) (y-z)^2 K(y-z,z)\phi(y) \int_{y-z}^y  \big( a'\big)_-(\xi)\,d\xi\,dz\,dy\\
& = \int_0^{\infty} \big(a'\big)_-(\xi) \int_{\xi}^{\xi+1} f(y)\phi(y) \int_{y-\xi}^1 f(z) (y-z)^2 K(y-z,z)\,dz\,dy\,d\xi\\
& \leq C \int_0^{\infty} \big(a'\big)_-(\xi)  \int_{\xi}^{\xi+1} y^{2+\alpha} (y-\xi)^{-\alpha} f(y) \phi(y)\,dy\,d\xi
 \end{split}
\]
and this implies \eqref{rigorous3}.

Strictly speaking in the previous argument we have been adding and subtracting a term that is possibly infinity due to divergences as $z \to 0$.
This difficulty can be removed, by integrating $z$ over the interval $(\eps,1)$ first, performing all the operations and finally let $\eps \to 0$.
\end{proof}

\ignore{
\begin{corollary}\label{C.lowerasymp2}
 There exists $R>0$ such that 
\[
 \int_{x^*}^{x^*+1} (a')_{-}(x)\,dx \leq \frac{C}{x^*} \qquad \mbox{ for all } x^*\geq R\,.
\]
As a consequence, we obtain
\[
a(x)-a(x^*) \geq - \frac{C}{x^*} \big( 1+(x-x^*)\big) \qquad \mbox{ for all } x \in [x^*,2x^*]\,. 
\]
\end{corollary}
\begin{proof}
We take a sequence $(R_n)$ as in Corollary \ref{C.lowerasymp1}. Then, by Lemma \ref{L.lowerasymp2} we have
\[
 \int_{x^*}^{x^*+1} (a')_{-}(x)\,dx \leq \frac{C}{R_n}  \qquad \mbox{ for all } x^* \in [R_n,2R_n]\,.
\]
Hence
\[
  \int_{x^*}^{x^*+1} (a')_{-}(x)\,dx \leq \frac{C}{x^*}  \qquad \mbox{ for all } x^* \in [R_n,2R_n]\,.
\]
Since $2R_n \geq (1+\eps) R_n \geq R_{n+1}$ we can iterate this estimate and obtain the desired statement. 

\end{proof}
}

\subsection{The limit $\frac{1}{x} \log \frac{1}{f(x)}$ exists}

\begin{lemma}\label{L.limit1} (Doubling Lemma)
There exist $R,C>0$ such that for 
$ x \geq R$ we have 
\[
 a(X)-a(x) \leq \frac{\log\big( C(X +1)\big)}{X} \qquad \mbox{ for } X=2 x -2\,.
\]
\end{lemma}
\begin{proof}
 From Lemma \ref{L.lowerasymp1} we have
\[
 a(y) \leq a( x) + \frac{C}{x} \big( 1+ (x-y)\big) \qquad \mbox{ for } y \in [x/2, x]
\]
and the same inequality for $z \in [ x/2, x]$.

The lower bound on the kernel \eqref{Kassump2} implies that for sufficiently large $x$ we have
$K \geq k_0$ in $[x-1,x]^2$. Then we obtain, using 
\eqref{eq1b}, for $X=2x-2$ 
that
\begin{equation}\label{doubling1}
 \begin{split}
  1& \geq \frac{C}{X^2} \int_{y \in [x-1,x]} dy \int_{z \in [x -1, x]} dz \,y e^{X a(X) - ya(y) - z a(z)}\\
& \geq \frac{C}{X} \int_{y \in [ x-1, x]} dy \int_{z \in [ x -1, x]} dz \,e^{Xa(X) - (y+z) a(x) - \frac{C}{x}(y+z)}\\
& \geq \frac{C}{X} \int_{y \in [x-1,x]} dy \int_{z \in [x -1,x]} dz\, e^{X (a(X) - a(x)) -C}\,.
 \end{split}
\end{equation}
As a consequence, we find
\[
 e^{X (a(X) - a(x)) -C} \leq CX
\]
and the statement of the Lemma follows.
\end{proof}

We define
\[
 M_{\delta}(x) = \max_{\delta x \leq y \leq x} a(y)\,.
\]
We obviously have
\[
 \limsup_{x \to \infty} a(x) = \limsup_{x \to \infty} M_{\delta}(x)\,.
\]
For the following we assume that for some $\theta \in (0,1)$
\begin{equation}\label{Kassump3}
 K(y,x-y) \geq k_0 \qquad \mbox{ for } y \in [\theta x,(1-\theta)x]\,. 
\end{equation}
Observe that this follows from \eqref{Kassump2} for $\theta \geq (1-\kappa)/2$. 

\begin{lemma}\label{L.limit2} (Flatness Lemma)
 Given $\delta>0$  there exist for all $\eps>0$ and $\eta>0$ numbers $R$ and $\sigma$ such that the following holds: if $a(x) \geq M_{\delta}(x)-\sigma$
with $x \geq R$, then $a(y) \geq M_{\delta}(x)-\eta$ for all $y \in ((\theta+\eps) x, (1-(\theta+\eps))x]$. 
\end{lemma}
\begin{proof}
 By definition we have that $a(y) \leq M_{\delta}(x)$ for all $y \in [\delta x, x]$. Then \eqref{eq1b} implies that
\begin{equation}\label{flatness1}
  1\geq \frac{1}{x^2} \int_{\delta x}^x  \int_{x{-}y}^{\infty} K(y,z) y e^{xa(x)-yM_{\delta}(x) - z a(z)}\,dz\,dy\,.
\end{equation}
Assume that the statement of the Lemma is not true. Then, given $\eta>0$,  there exists $z^* \in [(\theta+\eps) x, 
(1-(\theta+\eps))x]$ such that $a(z^*) < M_{\delta}(x)-\eta$.
By Lemma \ref{L.lowerasymp1} there exists $x^*\in [z^*(1+\eps/4),(1+3\eps/4)z^*]=:I$ such that
\[
 a(z) \leq a(x^*) + \frac{C}{x^*} \leq a(z^*)+\frac{C}{x^*}(1+(z^*-x^*)) < 
 M_{\delta}(x)-\eta + \frac{C}{x^*}
 \]
  for all $z \in [x^*-1,x^*]\cap I$.

  Next, we notice that $[x^*-1,x^*] \subseteq [\theta x,(1-\theta)x]$ for sufficiently large $x$.
  This follows since
  \[
   x^* -1 \geq z^*(1+\eps/4) \geq(\theta+\eps)(1+\eps/4)x-1\geq\theta x
  \]
and
\[
 x^* \leq z^*(1+3\eps/4) \leq (1-(\theta+\eps))(1+3\eps/4)x \leq(1-\theta)x
\]
if $x$ is sufficiently large.
  
  In addition we have, if $z \geq x^*-1/2$ and if $x$ is sufficiently large 
  that $x-z+1 \leq \frac{1+\kappa}{1-\kappa}z$.
  
  Thus, we can estimate the right hand side of \eqref{flatness1}  further via
  
\begin{equation}\label{flatness2}
 \begin{split}
  1 & \geq \frac{c}{x} \int_{x^*-1/2}^{x^*}  \int_{  x-z}^{ x+1-z}  e^{x(M_{\delta}(x)-\sigma) - y M_{\delta}(x) 
  - z (M_{\delta}(x)-\eta)}\,dy\,dz \\
&\geq \frac{c}{x}\int_{x^*-1/2}^{x^*} e^{-x\sigma + \eta z}\int_{x-z}^{x+1-z} e^{M_{\delta}(x)(x-(y+z))}\,dy\,dz\\
& \geq \frac{c}{x} e^{-\alpha_2} e^{-x\sigma + \eta \theta x}\,.
 \end{split}
\end{equation}
Choosing $\sigma=\eta \theta/2$ we obtain $e^{\eta \theta x/2}\leq Cx$ which 
gives a contradiction for sufficiently large $x$.
\end{proof}

\begin{prop}\label{P.limit}
 If $\theta < 1/3$ or equivalently $\kappa >1/3$, the limit of $a(x)$ as $x \to \infty$ exists.
 \end{prop}
\begin{proof}
 We assume that
\[
 b^*:=\limsup_{x \to \infty} a(x) > a^*:=\liminf_{x \to \infty} a(x)\,
\]
and let $\eps=(b^*-a^*)/10$.

{\it Claim 1:} 

There exists $\beta>0$ such that if $x_n=2+\beta 2^n$, we have $a(x_n) \leq a^* + 2\eps$.

\medskip
Indeed, the doubling Lemma \ref{L.limit1} implies that we have $a(x_{n+1}) -a(x_n) \leq \frac{1}{\sqrt{x_n}}$. The definition of $a^*$ implies that there exists
$x_0$ such that $a(x_0) =a(2+\beta) < a^*+\eps$. 
Then $a(x_n)-a(x_0) = \sum_{i=1}^n a(x_i)-a(x_{i-1}) \leq \sum_{i}^n \frac{1}{\beta 2^i} \leq \frac{2}{\beta}$ and thus
 $a(x_n) \leq a^* + \eps + \frac{C}{\sqrt{x_0}} \leq a^* + 2\eps$ for sufficiently large $\beta$ which 
proves Claim 1.

\bigskip
{\it Claim 2:}

There exists $({\bar x}_n)$ such that $a({\bar x}_n) \geq M_{\delta}({\bar x}_n) - \sigma$ with $M_{\delta}(
{\bar x}_n) \geq b^* - \eps$.

\medskip
First notice that for sufficiently large $x$ we have $M_{\delta}(x) \leq b^*+ \frac{\sigma}{4}$.

Next, since $a$ is continuous 
there exists $x_n^*$ with $a(x^*_n) < b^*-\frac{\sigma}{2}$. We define ${\bar x}_n:= \inf \{ x>x^*_n\,:\, a({\bar x}_n) = b^*-\frac{\sigma}{2}\}$.
Then 
\[
a({\bar x}_n) = b^*-\frac{\sigma}{2} > M_{\delta}({\bar x}_n) - \frac{\sigma}{4}-\frac{\sigma}{2} > M_{\delta}({\bar x}_n)- \sigma\,.
\]
Since $M_{\delta}({\bar x}_n) \geq a({\bar x}_n)$ we have $M_{\delta}({\bar x}_n) \geq b^*- \frac{\sigma}{2} \geq b^*-\eps$. 

\bigskip
We can now apply the flatness Lemma \ref{L.limit2} to conclude that
\[
 a(y) \geq M_{\delta}({\bar x}_n) - \eta \geq M_{\delta}({\bar x}_n) - 2\eps \qquad \mbox{ for all } y \in 
 [(\theta+\eps) {\bar x}_n, (1-(\theta+\eps)){\bar x}_n]\,.
\]
If we can choose $m \in \N$ such that $x_m \in [(\theta+\eps) {\bar x}_n, (1-(\theta+\eps)){\bar x}_n]$  we obtain a contradiction.
This is possible if $(1-(\theta+\eps)) {\bar x}_n > 2(\theta +\eps){\bar x}_n$. Since $\eps>0$ is arbitrary, this 
is satisfied for $1-\theta >2\theta$, hence for $\theta <1/3$.
\end{proof}

\section{The prefactor }

\subsection{Existence of the prefactor} \label{S.preex}

The function  $u$ defined in \eqref{udef} satisfies the equation
\begin{equation}\label{uequation}
 x^2 u(x) = \int_0^x \int_{x{-}y}^{\infty} \, y K(y,z) u(y)u(z) e^{a^* (x-(y{+}z))}\,dz\,dy\,.
\end{equation}
As explained in the introduction we need to assume from now on that $K$ is uniformly bounded below, that is $K$ satisfies \eqref{Kassump2} with $\kappa=1$.

We start with a consequence of the proof of  Lemma \ref{L.ublargex} that gives that $u$ has infinite first moment.

\begin{lemma}\label{L.unonintegrable}
 The integral $\int_0^{\infty} x u(x)\,dx$ is not finite.
\end{lemma}
\begin{proof}
 This is just the observation that \eqref{uequation} implies
\[
 x^2 u(x) \leq \int_0^x \int_{x{-}y}^{\infty} \, y K(y,z) u(y)u(z)\,dz\,dy\,.
\]
If we assume that $\int_0^{\infty} x u(x)\,dx< \infty$, we can apply the proof of Lemma \ref{L.ublargex} to $u(x)$ and obtain exponential decay of $u$. But this contradicts
the definition of $a^*$. 
\end{proof}

\begin{lemma}\label{L.uupperbound}
For any $A\geq 1$ there exists a constant $C_A>0$ such that 
\[
 \frac{1}{R} \int_1^{AR} u(x)\,dx \leq C_A\, \qquad \mbox{ for } R \geq \frac{2}{a^*}\,. 
\]
\end{lemma}
\begin{proof}
 We denote $U(q)=\int_1^{\infty}  e^{-qx} u(x)\,dx$ which is well-defined for all $q>0$. Then \eqref{uequation} implies 
\begin{equation}\label{utransformequation}
-\partial_q U(q)=\int_0^{\infty} \int_0^{\infty} K(y,z) y u(y) u(z) e^{-a^*(y+z)}\Big\{ \int_{\max(1,y)}^{\max(1,y+z)} \frac{1}{x} e^{(a^*-q)x}\,dx \Big\}\,dy\,dz\,.
\end{equation}
Symmetrizing and using the uniform lower bound on $K$ we find for $q \leq a^*/2$ that
\begin{equation}\label{uprimestimate}
 \begin{split}
  -\partial_q U(q) & \geq \int_1^{\infty} \int_1^{\infty}K(y,z) u(y) u(z) e^{-a^*(y+z)}  \frac{y}{y+z} \frac{1}{a^*{-}q} \Big( e^{(a^*-q)(y+z)} - e^{(a^*-q)y}\Big)\,dy\,dz\\
& \geq c \int_1^{\infty} \int_1^{\infty} u(y) u(z) e^{-q(y+z)} \Big( 1 - \frac{y}{y+z} e^{-(a^*-q)z} - \frac{z}{y+z} e^{-(a^*-q)y}\Big)\,dy\,dz\\
& \geq c\, U^2(q)\,,
\end{split}
\end{equation}
where the last inequality follows since $  1 - \frac{y}{y+z} e^{-(a^*-q)z} - \frac{z}{y+z} e^{-(a^*-q)y }\geq c>0$ for
$y,z \geq 1$ and $q \in [0,a^*/2]$. Integrating  inequality \eqref{uprimestimate} we find
\[
 \frac{1}{U(q)} \leq \frac{1}{U(q_0)} + c(q-q_0) \qquad \mbox{ for } 0<q < q_0 \leq \frac{a^*}{2}\,.
\]
If  $1/U(q_0) - cq_0 <0$ it follows that there exists $0<q^*\leq q_0$ such that $1/U(q^*) \leq 0$ which gives a contradiction to the fact that $U$ is well-defined for all $q>0$.
Since $q_0\leq \frac{a^*}{2}$ was arbitrary it follows that $U(q) \leq C/q$ for all $q\leq a^*/2$ and thus
\[
 \int_1^{AR} e^{-qx} u(x)\,dx \leq \frac{C}{q} \qquad \mbox{ for all } q \in \Big (0,\frac{a^*}{2}\Big]\,.
\]
Choosing $q=\frac{1}{R}$ the statement of the Lemma follows.
\end{proof}

\begin{corollary}\label{C.urlimit}
 Let $u_R(x):=u(Rx)$. Then there exists a sequence $R_j$ with $R_j \to \infty$ and a nonnegative measure $\mu \in {\cal M}^+(0,\infty)$ such that
$u_{R_j} \wto \mu$ on any compact subset of $(0,\infty)$. Furthermore we have
\begin{equation}\label{muapriori1}
 \int_{(0,\infty)} e^{-qy}\mu(y)\,dy \leq \frac{C}{q} \qquad \mbox{ for all } q>0
\end{equation}
and consequently 
\begin{equation}\label{muapriori1b}
 \int_{(0,x]}\mu(y)\,dy \leq Cx \qquad \mbox{ for all } x>0\,.
\end{equation}
\end{corollary}

\bigskip
\begin{lemma}\label{L.munontrivial}
 There exists a constant $C>0$ such that $\mu$ satisfies
\begin{equation}\label{muapriori1c}
 C x \leq \int_{(0,x]}\mu(x)\,dx  \qquad \mbox{ for all } x >0\,.
\end{equation}
\end{lemma}
\begin{proof}
 Our goal is to derive a uniform lower bound on averages of $u$ via uniform lower bounds on $U$. 
To that aim we define $V(q):= \int_1^{\infty} x^{-\alpha} u(x) e^{-qx}\,dx$.

{\it Step 1:} We claim for $U$ as in Lemma \ref{L.uupperbound}, that 
\begin{equation}\label{munontrivial1}
 -\partial_q U(q) \leq K_1 \Big( U V^{\frac{1}{1{-}\alpha}} + U + 1\Big)\qquad \mbox{ for } q \in\Big(0, \frac{a^*}{2}\Big]\,.
\end{equation}
We recall \eqref{utransformequation} from Lemma \ref{L.uupperbound} and split the integral on the right hand side into the following parts:
\[
\begin{split}
\int_1^{\infty} \,dy \int_1^{\infty} \,dz \dots +
 &\int_0^1\,dy\int_1^{\infty}\,dz\dots + \int_1^{\infty} \,dy \int_0^1 \,dz \dots + \int_0^1\,dy\int_0^1\,dz\dots \\
& =: (I)+(II)+(III)+(IV)\,.
\end{split}
\]
Using H\"olders inequality, we find
\begin{equation}\label{munontrivial2}
\begin{split}
|(I)|&\leq C \int_1^{\infty} \int_1^{\infty}\Big ( \Big(\frac{y}{z}\Big)^{\alpha} + \Big(\frac{z}{y}\Big)^{\alpha}\Big) u(y)u(z)  e^{-q(y+z)} \,dy\,dz \\
&\leq C V(q)  \int_1^{\infty} z^{\alpha} u(z) e^{-qz}\,dz \\
& \leq |\partial_q U|^{\alpha} |U|^{1{-}\alpha} |V| \leq \eps |\partial_q U| + C_{\eps} |U| |V|^{\frac{1}{1{-}\alpha}}\,.
\end{split}
\end{equation}
Next, using \eqref{ubsmallx}, we have
\begin{equation}\label{munontrivial3}
\begin{split}
  |(II)| &\leq C \int_0^1 \int_1^{\infty}  y^{1{-}\alpha} u(y)  z^{\alpha} u(z)e^{-a^*(y+z)}\int_1^{y+z} \frac{e^{(a^*-q)x}}{x}\,dx \,dz\,dy\\
&\leq C \int_1^{\infty} z^{\alpha}u(z) \frac{e^{-qz}}{z}\,dz \leq C U\,.
\end{split}
\end{equation}
Similarly we estimate
\begin{equation}\label{munontrivial4}
 \begin{split}
  |(III)| & \leq \int_1^{\infty} \int_0^1 y^{1+\alpha} z^{-\alpha} u(y)u(z) e^{-a^*(y+z)}\int_y^{y+z} \frac{e^{(a^*-q)x}}{x}\,dx \,dz\,dy\\
& \leq C \int_1^{\infty} \int_0^1 y^{1{+}\alpha} z^{-\alpha} u(y)u(z) e^{-a^*(y+z)} \frac{z}{y} e^{(a^*-q)y}\,dz\,dy\\
& \leq C \int_1^{\infty} y^{\alpha} u(y) e^{-qy}\,dy\\
& \leq \eps |\partial_q U(q)| + C_{\eps} |U(q)|\,.
\end{split}
\end{equation}
Finally, we have
\begin{equation}\label{munontrivial5}
 \begin{split}
  |(IV)|& \leq C \int_0^1 \int_0^1 u(y) u(z) y^{1-\alpha} z^{-\alpha} \Big\{\int_1^{\max(1,y+z)} \frac{e^{(a^*-q)x}}{x}\,dx \Big\}e^{-a^*(y+z)} \,dy\,dz \\
& \leq C \int_0^1 \int_0^1 u(y) u(z) y^{1-\alpha} z^{1-\alpha}  \,dy\,dz \leq C
 \end{split}
\end{equation}
and \eqref{munontrivial1} follows from \eqref{munontrivial2}-\eqref{munontrivial5}, noticing that $|\partial_q U(q)|=-\partial_q U(q)$.

\medskip
{\it Step 2:}
There exists $K_2>0$ such that 
\begin{equation}\label{vestimate}
 V(q) \leq \frac{K_2}{q^{1-\alpha}}\qquad \mbox{ for all } q \in \Big(0,\frac{a^*}{2}\Big]\,.
\end{equation}
Indeed, the definition of $V$ implies that for $q>\hat q>0$ 
\begin{equation}\label{vestimate1}
 \begin{split}
  |V(q)-V(\hat q)| & \leq \Big| \int_1^{\infty} y^{-\alpha} u(y) \Big( e^{-qy} - e^{-\hat q y}\Big)\,dy \Big|\\
& \leq |q-\hat q|^{\alpha} \Big| \int_1^{\infty} ((q-\hat q)y)^{-\alpha} \Big( 1 - e^{-y(q-\hat q)}\Big) e^{-\hat q y} u(y)\,dy\\
& \leq K_3 |q-\hat q|^{\alpha} U(\hat q)\,.
 \end{split}
\end{equation}
Recall that we have proved in Lemma \ref{L.uupperbound} that $U(q) \leq C/q$ for $q \in (0,a^*/2]$. 
Thus, we obtain for $\hat q=2^{-(n+1)}$ and $q=2^{-n}$ that
\[
  V\big(2^{-(n+1)}\big) \leq V(2^{-n}) + C 2^n 2^{-(n+1)\alpha} 
\]
and iterating this formula we find
\[
 V\big(2^{-(n+1)} \big)\leq C \Big ( 1 + \sum_{i=1}^n 2^{i(1-\alpha)}\Big) \leq C 2^{(n+1)(1-\alpha)}
\]
and this proves \eqref{vestimate}.

\medskip
{\it Step 3:}
We claim that there exists $C>0$ such that 
\begin{equation}\label{munontrivial6}
 U(q) \geq \frac{C}{q} \qquad \mbox{ for all } q \in \Big(0,\frac{a^*}{2}\Big]\,.
\end{equation}
Suppose that \eqref{munontrivial6} is not true. Then, given an arbitrarily small $\eps>0$, there exists $q^*\in \Big(0,\frac{a^*}{2}\Big]$ such that $U(q^*) \leq \frac{\eps}{q^*}$. 
We can assume without loss of generality that $U \geq 1$. 

Integrating \eqref{munontrivial1} and using \eqref{vestimate} we find 
\begin{equation}\label{mu2a}
 U\Big( \frac{q}{2}\Big) \leq U(q) \exp \Big( K_1 {K_2}^{\frac{1}{1-\alpha}} \ln 2\Big)
\end{equation}
for all $q \in \Big(0,\frac{q^*}{2}\Big]$. For the following let $K:=\max\Big(K_1,K_2,K_3, \exp\big(K_1 K_2^{\frac{1}{1-\alpha}} \ln 2 \big)\Big)$. 

It follows in particular from \eqref{mu2a} that 
\begin{equation}\label{mu2}
U\Big( \frac{q^*}{2}\Big) \leq \frac{1}{\frac{q^*}{2}} \frac{\eps}{2} K\,.
\end{equation}
Estimate \eqref{vestimate1} also implies that
\begin{equation}\label{mu2b}
 V\Big( \frac{q^*}{2}\Big) \leq V(q^*) + K \big( \frac{q^*}{2}\big)^{\alpha} U\big( \frac{q^*}{2}\big) \leq \Big( \frac{q^*}{2}\Big)^{\alpha-1} \Big( \frac{K}{2^{1-\alpha}} + K^2 \frac{\eps}{2}
 \Big)\,.
\end{equation}
We can find $\theta>1$  such that 
for sufficiently small $\eps$ 
\begin{equation}\label{theta2}
  \frac{K}{2^{1-\alpha}} + \frac{\eps}{2} K^2
 \exp \Big( K\sum_{i=0}^{\infty} \theta^{-\frac{i}{1-\alpha}} \Big) \leq \frac{K}{\theta}\,.
\end{equation}
We can now prove by induction that
\begin{equation}\label{V1}
 V\Big(\frac{q^*}{2^{n}}\Big) \leq \frac{K}{\theta^n} \Big( \frac{q^*}{2^n}\Big)^{\alpha-1}
\end{equation}
and
\begin{equation}\label{U1}
 U\Big( \frac{q^*}{2^n}\Big) \leq \Big( \frac{q^*}{2^n}\Big)^{-1} \frac{\eps}{2^n} \exp\Big( K\sum_{i=0}^{n} \theta^{-\frac{i}{1-\alpha}}\Big)\,.
\end{equation}
In fact, \eqref{U1} for $n=1$ follows from \eqref{mu2} and \eqref{V1}  from \eqref{mu2b} and \eqref{theta2}. To go from $n$ to $n+1$  we 
we use \eqref{vestimate1}, \eqref{mu2a}  and
\eqref{theta2}  to obtain, assuming $\theta <2$, that
\begin{equation}\label{V2}
\begin{split}
 V\Big(\frac{q^*}{2^{n+1}}\Big)&\leq  V\Big(\frac{q^*}{2^n}\Big) + K\Big(\frac{q^*}{2^{n+1}}\Big)^{\alpha} U\Big(\frac{q^*}{2^{n}}\Big) \\
& \leq \frac{K}{\theta^n} \Big( \frac{q^*}{2^n}\Big)^{\alpha-1} + K^2 \Big( \frac{q}{2^{n+1}}\Big)^{\alpha-1} \frac{\eps}{2^{n+1}} \exp\Big( K \sum_{i=0}^{n} \theta^{-\frac{i}{1-\alpha}} \Big)\\
 &= \Big( \frac{q^*}{2^{n+1}}\Big)^{\alpha-1} \frac{1}{\theta^n} \Big( \frac{K}{2^{1-\alpha}} + \frac{\eps}{2} K^2 \exp \Big( K \sum_{i=0}^n \theta^{-\frac{i}{1-\alpha}} \Big)\\ 
&\leq \Big( \frac{q^*}{2^{n+1}}\Big)^{\alpha-1} \frac{K}{\theta^n}\,.
\end{split}
\end{equation}
Going back with $\eqref{V2}$ to \eqref{munontrivial1} we find
\begin{equation}\label{U2}
 \begin{split}
U\Big( \frac{q^*}{2^{n+1}}\Big) & \leq U\Big( \frac{q^*}{2^n}\Big) \exp \Big( K \theta^{-(n+1)/(1-\alpha)}\Big)\\
&  \leq\Big( \frac{q^*}{2^n}\Big)^{-1} \frac{\eps}{2^{n+1}} \exp \Big( K \sum_{i=0}^{n+1} \theta^{-i/(1-\alpha)}\Big) \,.
 \end{split}
\end{equation}
Thus, \eqref{V2} and \eqref{U2} prove \eqref{V1} and \eqref{U1}.
 In particular, we have a constant $C>0$ such that 
 $U(q^*/2^n) \leq \frac{C}{q^*}$ which gives a uniform bound on $U$. As  a consequence of \eqref{munontrivial1} we also get  a uniform bound on 
$|\partial_q U|$ and thus a contradiction to Lemma \ref{L.unonintegrable}. This finishes the proof of \eqref{munontrivial6}.

\medskip
{\it Step 4:} 
We show that
\begin{equation}\label{munontrivial8}
 \int_1^{BR} u(x)\,dx \geq cR
\end{equation}
for sufficiently large $B$ from which the claim of the Lemma follows. Indeed, choosing $q=\frac{1}{R}$ in \eqref{munontrivial6} gives
\[
 R \leq \int_1^{BR} e^{-\frac{x}{R}} u(x)\,dx + \int_{BR}^{\infty} e^{-\frac{x}{R}} u(x)\,dx\,.
\]
On the other hand, using Lemma \ref{L.uupperbound}, we have
\[
\begin{split}
 \int_{BR}^{\infty} e^{-\frac{x}{R}} u(x)\,dx &\leq \sum_{n=0}^{\infty} \int_{2^n BR}^{2^{n+1}BR}  e^{-\frac{x}{R}} u(x)\,dx\\
& \leq \sum_{n=0}^{\infty} e^{-2^nB}  \int_{2^n BR}^{2^{n+1}BR}u(x)\,dx
 \leq R \sum_{n=0}^{\infty} B 2^n  e^{-2^nB}\,.
\end{split}
\]
Since $\sum_{n=0}^{\infty} B 2^n  e^{-2^nB}<\frac 1 2$ for sufficiently large $B$  estimate \eqref{munontrivial8} follows.

finally, the lower bound in \eqref{muapriori1c} is a consequence of \eqref{munontrivial8}.
\end{proof}

\begin{lemma}\label{L.muequation}
 $\mu$ is a weak solution to \eqref{muequation}.
\end{lemma}
\begin{proof}
 Equation \eqref{muequation} means that for $\phi \in C_0([0,\infty))$ it holds
\begin{equation}\label{weakmuequation}
 \int_0^{\infty} x^2 \phi(x) \mu(x)\,dx = \frac{1}{a^*} \int_0^{\infty} \int_0^{\infty} K(x,y) y \mu(x)\mu(y) \phi(x+y)\,dx\,dy\,.
\end{equation}
We now take $\phi \in C_0^1([0,\infty))$. From \eqref{uequation} we deduce that $u_R$ satisfies
\begin{equation}\label{urequation}
\begin{split}
 x^2 u_R(x) &= \frac{1}{R^2} \int_0^{Rx} \int_{Rx-y}^{\infty} K(y,z) y u(y) u(z) e^{a^*(Rx-(y+z))} \,dz\,dy \\
& = R \int_0^x \int_{x-y}^{\infty} K(y,z) y u_R(y) u_R(z) e^{a^*R(x-(y+z))}\,dz\,dy\,.
\end{split}
\end{equation}
Hence, 
\begin{equation}\label{urweakequation}
 \int_0^{\infty} x^2 u_R(x)\phi(x)\,dx = \int_0^{\infty} \int_0^{\infty} K(y,z) y u_R(y) u_R(z) \psi_R(y,z) \,dy\,dz
\end{equation}
with
\begin{equation}\label{psidef}
 \psi_R(y,z) := R \int_y^{y+z} \phi(x) e^{a^*R(x-(y+z))}\,dx\,.
\end{equation}
The function $\psi_R$ satisfies
\begin{equation}\label{psi1}
 |\psi_R(y,z)| \leq \|\phi\|_{L^{\infty}} R z\,
\end{equation}
and
\begin{equation}\label{psi2}
 \psi_R(y,z) = \frac{1}{a^*}\int_0^{a^*Rz} \phi\Big( y+z - \frac{\xi}{a^*R}\Big)e^{-\xi}\,d\xi \to \frac{1}{a^*}\phi(y+z) \quad \mbox{ as } R \to \infty \mbox{ for all }y,z>0\,.
\end{equation}
Furthermore
\begin{equation}\label{psi3}
 \begin{split}
  \Big| \psi_R(y,z) - \frac{1}{a^*}\phi(y+z)\Big| & \leq \Big| \frac{1}{a^*} \int_0^{a^*Rz} \Big[\phi\Big(y+z - \frac{\xi}{a^*R}\Big) -\phi(y+z)\Big] e^{-\xi}\,d\xi\Big|\\
& \qquad + \Big| \int_{a^*Rz}^{\infty} \frac{\phi(y+z)}{a^*} e^{-\xi}\,d\xi\Big|\\
& \leq\|\phi'\|_{L^{\infty}} \int_0^{a^*Rz} \frac{\xi}{a^*R} e^{-\xi}\,d\xi + \|\phi\|_{L^{\infty}} e^{-a^*Rz}\\
& \leq \|\phi'\|_{L^{\infty}}  \frac{C}{R} + \|\phi\|_{L^{\infty}} e^{-a^*Rz}\,.
 \end{split}
\end{equation}

For a large constant $L$ 
we split the integral on the right hand side of \eqref{urweakequation} into the parts
\[
 \int_0^{\infty}\,dy \int_0^{L/R} \,dz \dots + \int_0^{\infty}\,dy \int_{L/R}^{\infty}\,dz \dots =:(I)+(II)\,.
\]
We recall that due to \eqref{ubsmallx} we have $\int_0^L x^{1-\alpha} u(x)\,dx \leq C_L$. Hence
\begin{equation}\label{psi4}
 \int_0^{L/R} x^{1-\alpha} u_R(x)\,dx = R^{\alpha-2} \int_0^L x^{1-\alpha} u(x)\,dx \leq \frac{C_L}{R^{2-\alpha}}\,.
\end{equation}
Second, Lemma \ref{L.uupperbound} implies  $\int_{L/R}^A u_R(x)\,dx \leq C$. Hence $u_R \chi_{\{x \geq L/R\}} \wto \mu$ on
$[0,A]$.  Combining this with  properties \eqref{psi1}-\eqref{psi3} and \eqref{Kassump1} we find
\begin{equation}\label{urweak1}
 |(I)|\leq C R \int_0^{L/R} u_R(z) z^{1-\alpha} \,dz \leq \frac{C_L}{R^{1-\alpha}} \to 0 \qquad \mbox{ as } R \to \infty \mbox{ for fixed }L\,
\end{equation}
and
\begin{equation}\label{urweak2}
 \limsup_{R\to\infty} \Big | (II) - \int_0^{\infty} \int_0^{\infty} K(y,z) y \mu(y)\mu(z)\frac{1}{a^*}\psi(y+z)\,dy\,dz\Big| \leq C e^{-a^*L}\,.
\end{equation}
Since $L$ was arbitrary the desired result follows.
\end{proof}

\subsection{Uniqueness of the prefactor}\label{S.preuni}

In this section we are going to show that under additional assumptions on the kernel $K$ the solution to \eqref{muequation} is unique.

\begin{lemma}\label{L.profunique}
Let $K$ satisfy the assumptions of Theorem \ref{T4}.
Then, if $\eps>0$ is sufficiently small, the only solution to \eqref{muequation} that satisfies  \eqref{muapriori1b} and \eqref{muapriori1c} is 
the constant one, i.e. $\mu \equiv \mu_0:=2a^*/\int_0^1 K(s,1-s)\,ds$.
\end{lemma}

\begin{proof}
 
Our proof is somewhat similar in spirit, though simpler, as the proof of uniqueness of self-similar solutions to the coagulation equation with kernels satisfying
\eqref{Kassump4} and \eqref{Kassump5}. We can assume without loss of generality that $2 a^*=1$ such that $|\mu_0-1|\leq C \eps$. 
We introduce the Laplace-transform of $\mu$ via
\[
 \tilde U(q)=\int_{(0,\infty)} e^{-qx}\mu(x)\,dx\,, \qquad q>0\,.
\]
Using the usual dyadic decomposition we can derive from \eqref{muapriori1b} that
\begin{equation}\label{muapriori2}
 \int_0^x y^{-\alpha} \mu(y)\,dy \leq C x^{1-\alpha}\qquad \mbox{ for all } x>0\,.
\end{equation}
Furthermore, $\tilde U$
satisfies the equation
\begin{equation}\label{Uequation}
 -\partial_q\tilde U= {\tilde U}^2 + \int_0^{\infty} \int_0^{\infty} \Big(K(x,y)-1\Big) e^{-q(x+y)} \mu(x)\mu(y)\,dx\,dy=: {\tilde U}^2 +{\cal M}\,.
\end{equation}

We first estimate ${\cal M}$.
By symmetry, using assumption \eqref{Kassump4}, we have
\[
 {\cal M} \leq C \eps \int_0^{\infty} \int_0^{x/2} \Big( \frac{x}{y}\Big)^{\alpha} \mu(x-y) \mu(y) e^{-qx}\,dy\,dx\,.
\]
For any  $R>0$ we estimate, using \eqref{muapriori1b} and \eqref{muapriori2}, 
\[
\begin{split}
 \int_R^{2R} \int_0^{x/2} \Big(\frac{x}{y}\Big)^{\alpha} \mu(x-y)\mu(y)\,dy\,dx
& = \int_0^R \int_{\max(R,2y)}^{2R} \Big(\frac{x}{y}\Big)^{\alpha} \mu(x-y)\mu(y)\,dx\,dy\\
& \leq C R^{\alpha} \int_0^R \frac{\mu(y)}{y^{\alpha}}\,dy \int_{0}^{2R} \mu(x)\,dx  \leq C R^2\,.
\end{split}
\]
and this implies
\begin{equation}\label{uni2}
 \begin{split}
{\cal M}& \leq C \eps \sum_{n=-\infty}^{\infty} \int_{2^n}^{2^{n+1}} e^{-qx} \int_0^{x/2}   \Big(\frac{x}{y}\Big)^{\alpha} \mu(x-y)\mu(y)\,dy\,dx\\
& \leq C\eps \sum_{n=-\infty}^{\infty} e^{-q2^n} \big( 2^n\big)^2 \leq \frac{C \eps}{q^2}\,.
 \end{split}
\end{equation}
Next, we claim that
\begin{equation}\label{Wapriori}
\Big|\tilde U(q)-\frac{1}{q}\Big| \leq \frac{C\eps}{q} \qquad \mbox{ for all } q>0\,.
\end{equation}
Suppose otherwise. Then there exists for any arbitrarily large $K$ a number  $q^*>0$ such that 
$|\tilde U(q^*)-\frac{1}{q^*}|> \frac{K\eps}{q^*}$. Define $G(x):=q^*U(q^*x)$ such that
\[
 -\partial_x G(x) = G(x)^2 + O\Big(\frac{\eps}{x^2}\Big) \qquad \mbox{ and } |G(1)-1|\geq K\eps\,.
\]
Let $\bar G$ be the solution of $-\partial_x \bar G = {\bar G}^2$ with $\bar G(1)=G(1)$, that 
is $\bar G (x) = \frac{G(1)}{1+G(1)(x-1)}$ and $\bar G(1/2)= \frac{2G(1)}{2-G(1)}$. By the standard theory of 
differential equations we find $|G(1/2) - {\bar G}(1/2)|\leq C\eps$ and thus
\[
 \big| G(1/2) -2\big| \geq \Big| \frac{2G(1)}{2-G(1)} -2\Big| - C\eps = \frac{4}{|2-G(1)|}|G(1)-1| - C\eps\,.
\]
If $G(1) \geq 2$ we get immediately a contradiction to the fact that $\tilde U(q)$ is defined for all $q>0$. On the other hand,
we know  that $G(1) \geq c_1>0$. Hence we obtain that
\[
 \big| G(1/2) -2\big| \geq \theta |G(1)-1| \qquad \mbox{ for some } \theta>2\,,
\]
if $K$ is sufficiently large. 
Thus, we have obtained
\[
 \Big| U \Big( \frac{q^*}{2}\Big) - \frac{1}{\frac{q^*}{2}}\Big| >\frac{\theta}{2} \frac{K\eps}{\frac{q^*}{2}}\,.
\]
We can iterate the argument to get
\[
 \Big| U \Big( \frac{q^*}{2^n}\Big) - \frac{1}{\frac{q^*}{2^n}}\Big| > \Big(\frac{\theta}{2}\Big)^n\frac{K\eps}{\frac{q^*}{2^n}}\,,
\]
which gives again a contradiction since $\theta/2>1$. Hence, we have proved \eqref{Wapriori}.

\medskip
To proceed we introduce
\[
 \nu(x):=\mu(x) - \mu_0 \qquad \mbox{ with } \qquad \mu_0 = \frac{1}{\int_0^1 K(s,1-s)\,dx}
\]
and
\[
 V(q)=\int_0^{\infty} e^{-qx}\nu(x)\,dx\,.
\]
Since 
\[
 x\nu(x) = 2\mu_0 \int_0^x K(y,x-y)\nu(y)\,dy + \int_0^x K(y,x-y) \nu(x-y)\nu(y)\,dy
\]
we find
\[
\begin{split}
 -\partial_q V& = \frac{2\mu_0}{q} V + 2 \mu_0 \int_0^{\infty} e^{-qx} \int_0^x \Big( K(y,x-y)-1\Big)\nu(y)\,dy\,dx\\
&\qquad + 
\int_0^{\infty} e^{-qx} \int_0^x K(y,x-y) \nu(x-y)\nu(y)\,dy\,dx\,.
\end{split}
\]
Integrating this equation gives
\begin{equation}\label{uni3a}
\begin{split}
 q^{2\mu_0} V(q)& = -2 \mu_0 \int_0^q \eta^{2\mu_0}  \int_0^{\infty} e^{-\eta x} \int_0^x \Big(\big( K(y,x -y) -1\big)\nu(y)\\
&\qquad\qquad + K(y,x-y) \nu(x-y)\nu(y)\Big) \,dy\,dx\,d\eta\,.
\end{split}
\end{equation}
We introduce the norm
\begin{equation}\label{normdef}
 \|V\|:=\sup_{q>0} | qV(q)|
\end{equation}
such that $|\mu_0-1|\leq C\eps$ and \eqref{Wapriori} imply $\|V\| \leq C\eps$. Next, we claim that
\begin{equation}\label{uni3}
 \sup_{q>0} \Big| q \int_0^q \Big(\frac{\eta}{q}\Big)^{2\mu_0}  \int_0^{\infty} e^{-\eta x} \int_0^x \big( K(y,x-y) -1\big)\nu(y)\,dy\,dx\,d\eta \Big| \leq C \eps \|V\|
\end{equation}
and
\begin{equation}\label{uni4}
 \sup_{q>0} \Big| q \int_0^q \Big( \frac{\eta}{q}\Big)^{2\mu_0}  \int_0^{\infty} e^{-\eta x } \int_0^x  K(y,x-y) \nu(x-y)\nu(y)\,dy\,dx\,d\eta \Big| \leq C \eps \|V\|^2\,.
\end{equation}
Then, \eqref{uni3}, \eqref{uni3a} and  \eqref{uni4} imply 
\[
 \| V\| \leq C \eps \|V\| + C \|V\|^2 
\]
 and since $\|V\|\leq C\eps$ it follows that  $V=0$ if $\eps$ is sufficiently small.

It remains to prove \eqref{uni3} and \eqref{uni4} which we will do by contradiction. By scaling we can assume that $\eps=1$. To prove \eqref{uni3} we thus  assume that there
exist  sequences $(q_n), (\nu_n)$ and $(K_n)$ such that $|q_n V(q_n)| \to 0$ as $n \to \infty$ while
\begin{equation}\label{uni5}
\Big|  q_n \int_0^{q_n} \Big(\frac{\eta}{q_n}\Big)^{2\mu_0}  \int_0^{\infty} e^{-\eta x} \int_0^x \big( K_n(y,x-y) -1\big)\nu_n(y)\,dy\,dx\,d\eta\Big| \geq 1\,.
\end{equation}
 By assumption \eqref{Kassump5} we can assume that there exists $K=K(x,y)$ such that $K_n \to K$ locally uniformly on $(0,\infty)^2$ and $K$ satisfies the same bounds as $K_n$. 
We change variables in \eqref{uni5}, such that $\hat x=q_n x$, $\hat y = q_n y$ and $\hat \eta = \eta/q_n$, define $\hat \nu_n(\hat y)=\nu_n(y)$ and obtain, dropping the hats in the notation, that
\eqref{uni5} turns into
\begin{equation}\label{uni6}
 \Big|\int_0^1 \eta^{2\mu_0} \int_0^{\infty} e^{-\eta x} \int_0^x \big(K_n(y,x-y)-1\big)\nu_n(y)\,dy\,dx\,d\eta\Big | \geq 1\,. 
\end{equation}
Changing the order of integration \eqref{uni6} gives
\begin{equation}\label{uni7}
 \Big| \int_0^{\infty} \nu_n(y) \int_y^{\infty} \int_0^1 e^{-\eta x}\eta^{2\mu_0} \big( K_n(y,x-y)-1\big) \,d\eta \,dx\,dy\Big| \geq 1\,.
\end{equation}
If we define $\phi(x):=\int_0^1 e^{-\eta x} \eta^{2\mu_0}\,d\eta$ we have that $\phi \in C^0(0,\infty) \cap L^{\infty}(0,\infty)$ with $\phi(x) \leq \frac{C}{x^{1+2\mu_0}}$ as $x \to \infty$. The bound \eqref{Kassump4}
then implies that for $\psi_n(y):=\int_y^{\infty} \phi(x) \big( K_n(y,x-y)-1\big) \,dx$ we have
$\psi_n(y) \leq \frac{C}{y^{\alpha}}$ for $y \geq 1$, $\psi_n(y) \leq \frac{C}{y^{2\mu_0}}$ for $y \geq 1$ and that $\psi_n\to \psi$ locally uniformly on $(0,\infty)$.

By assumption $q_n V(q_n) \to 0$. This means, after using the rescaling above, that $\nu_n \wto 0$ as $n \to \infty$. The bounds on $\nu_n$ and $\psi_n$ and the fact that $2 \mu_0 \geq 2-C\eps$ 
imply that $\int_0^{\infty} \psi_n(y)\nu_n(y)\,dy \to 0$ which gives a contradiction to \eqref{uni7} and finishes the proof of \eqref{uni3}.

The proof of \eqref{uni4} follows similarly and we omit the details here.
\end{proof}

\bigskip
{\bf Acknowledgment.} The authors acknowledge support through the CRC 1060 {\it The mathematics of emergent effects } at the University of Bonn, that is funded through the
German Science Foundation (DFG).

{\small
\bibliographystyle{alpha}%
\bibliography{../coagulation}%
}

\end{document}